\documentclass[11pt]{amsart}
\usepackage{amssymb}
\usepackage{graphicx}
\usepackage{cancel}
\usepackage{xcolor} 
\usepackage{tensor}
\usepackage{fullpage} 
\usepackage{amsmath}
\usepackage{amsthm}
\usepackage{verbatim}
\usepackage[colorlinks=true, allcolors=blue]{hyperref}
\usepackage{xfrac}
\usepackage[notextcomp]{stix} 

\usepackage{grffile}  
\usepackage{subfig} 
\usepackage{tabu} 
\usepackage{svg}  
\graphicspath{{./pictures/}} 
\DeclareMathAlphabet\mathbfcal{OMS}{cmsy}{b}{n} 
\usepackage{bbm}

\usepackage{multirow} 
\usepackage{array} 
\newcolumntype{L}{>{$}l<{$}} 
\newcolumntype{D}{>{$\displaystyle}l<{$}} 
\newcolumntype{C}{>{$}c<{$}} 

\usepackage{enumitem}
\setlist[itemize]{leftmargin=1.5em}

\usepackage{seqsplit}

\definecolor{green}{rgb}{0,0.8,0} 
\definecolor{greenr}{rgb}{0.6,0.8,0}

\newtheorem{theorem}{Theorem}[section]
\newtheorem{corollary}[theorem]{Corollary}
\newtheorem{lemma}[theorem]{Lemma}
\newtheorem{proposition}[theorem]{Proposition}

\newtheorem{definition}[theorem]{Definition}

\theoremstyle{remark}
\newtheorem{remark}[theorem]{Remark}

\numberwithin{equation}{section}

\newcommand{\nnrm}[1]{{\vert\kern-0.25ex\vert\kern-0.25ex\vert #1 
		\vert\kern-0.25ex\vert\kern-0.25ex\vert}}

\def\d{\,\mathrm{d}}
\newcommand{\Emod}{\calE^\text{mod}}

\newcommand{\calA}{\mathcal A}

\newcommand{\calD}{\mathcal D}
\newcommand{\calE}{\mathcal E}
\newcommand{\calF}{\mathcal F}

\newcommand{\calJ}{\mathcal J}
\newcommand{\calK}{\mathcal K}
\newcommand{\calL}{\mathcal L}
\newcommand{\calM}{\mathcal M}

\newcommand{\calR}{\mathcal R}

\newcommand{\calX}{\mathcal X}
\newcommand{\calY}{\mathcal Y}

\newcommand{\frkm}{1}

\newcommand{\R}{\mathbb R}

\newcommand{\pa}{\partial}

\newcommand{\e}{\ep}
\newcommand{\lt}{\left}
\newcommand{\rt}{\right}
\newcommand{\bq}{\begin{equation}}
	\newcommand{\eq}{\end{equation}}

\newcommand{\Om}{\Omega}

\newcommand{\levert}{\left\vert}
\newcommand{\rivert}{\right \vert}
\newcommand{\levertl}{\left\Vert}
\newcommand{\rivertl}{\right \Vert}

\newcommand{\eqh}[1]{\begin{equation*}
		\begin{split}
			#1
		\end{split}
\end{equation*}}
\newcommand{\equ}[1]{\begin{equation}
		\begin{split}
			#1
		\end{split}
\end{equation}}
\newcommand{\equh}[1]{\begin{equation*}
		\begin{split}
			#1
		\end{split}
\end{equation*}}
\newcommand{\vr}{\rho}
\newcommand{\ep}{\varepsilon}
\newcommand{\dx}{{\rm{d}}x}
\newcommand{\dy}{{\rm{d}}y}
\newcommand{\dt}{{\rm{d}}t}
\newcommand{\dxdt}{\dx\,\dt}
\newcommand{\px}{\partial_x}
\newcommand{\pt}{\partial_t}
\newcommand{\lr}[1]{\left( #1 \right)}
\newcommand{\intO}[1]{\int_\R{#1}\, \dx}
\newcommand{\intM}[1]{\int_{\Omega_\ep}{#1}\, \dx}
\newcommand{\iintO}[1]{\int\!\!\!\int_{\R\times \R}{#1}\, \dy\,\dx}
\newcommand{\iintM}[1]{\int\!\!\!\int_{\Omega_\ep\times\Omega_\ep}{#1}\, \dy\,\dx}

\newcommand{\intTO}[1]{\int_0^T\!\!\!\int_\R{#1}\, \dx\, \dt}
\newcommand{\intTM}[1]{\int_0^T\!\!\!\int_{\Omega_\ep}{#1}\, \dx\, \dt}

\newcommand{\Dt}{\frac{ \rm d}{{\rm d}t}}

\newcommand{\intOy}[1]{\int_\R{#1}\, \dy}
\newcommand{\phie}{{\phi_\ep}}
\newcommand{\tw}{{\widetilde{W}}}
\newcommand{\gmhalf}{{\gamma-\frac{1}{2}}}
\newcommand{\ds}{\displaystyle}
\newcommand{\tvr}{\widetilde{\vr}}

\begin{document}
	
	\title{Existence of weak solutions and long-time asymptotics for hydrodynamic model of swarming }
	
	\author{Nilasis Chaudhuri}
	\address{Faculty of Mathematics, Informatics, and Mechanics, University of Warsaw, ul. Banacha 2, Warsaw 02-097, Poland}
	\email{nchaudhuri@mimuw.edu.pl}
	
	\author{Young-Pil Choi}
	\address{Department of Mathematics, Yonsei University, 50 Yonsei-Ro, Seodaemun-Gu, Seoul 03722, Republic of Korea}
	\email{ypchoi@yonsei.ac.kr}
	
	\author{Oliver Tse}
	\address{Department of Mathematics and Computer Science, Eindhoven University of Technology, 5600 MB Eindhoven, The Netherlands}
	\email{o.t.c.tse@tue.nl}

	\author{Ewelina Zatorska}
	\address{Mathematics Institute, Zeeman Building, University of Warwick. Coventry CV4 7AL, United Kingdom}
	\email{ewelina.zatorska@warwick.ac.uk}

	\date{\today}
	
	
	
	\maketitle
	
	
	\begin{abstract}
		We consider a one-dimensional hydrodynamic model featuring nonlocal attraction-repulsion interactions and singular velocity alignment. We introduce a two-velocity reformulation and the corresponding energy-type inequality, in the spirit of the Bresch-Desjardins estimate. We identify a dependence between the communication weight and interaction kernel and between the pressure and viscosity term allowing for this inequality to be uniform in time. It is then used to study long-time asymptotics of solutions.
		
	\end{abstract}
	

	%
	%
	%
	%
	%
	
	\section{Introduction}
	The purpose of this manuscript is to study the existence of solutions and the long-time behavior of the one-dimensional hydrodynamic model of collective motion with unknown density $\rho$ and velocity $u$ satisfying the conservation of mass and momentum balance equations
	\begin{subequations}\label{main_eq}
		\begin{align}
			&\pa_t \rho + \pa_x (\rho u) = 0, \quad x \in \R, \quad t > 0, \label{main_eq1}\\
			&\pa_t (\rho u) + \pa_x (\rho u^2) = - \rho \pa_x \delta\calF(\rho) + \pa_x (\mu(\rho)\pa_x u) + \rho [\phi *, u]\rho - \tau\rho u,\label{main_eq2}
		\end{align}
	\end{subequations}
 henceforth, simply called the Navier-Stokes system. In the model above, $\rho\mapsto\calF(\rho)$ is the driving functional with contribution due to the internal energy and the nonlocal interaction parts, i.e.
	\[
	\calF(\rho) = \int_\R \int_0^{\rho(x)} \varphi(r)\, {\rm d} r\,\dx + \frac12 \int_\R \rho(x) W * \rho(x)\,\dx,
	\]
	and $\delta \calF$ is its $L^2$-variational derivative given by
	\equ{\label{dF}
	\delta \calF(\rho) = \varphi(\rho) + W * \rho.
	}
Here, we consider the explicit form of $\varphi:\R_+ \to \R_+$, $\R_+\coloneq[0,\infty)$, given by
 \equ{\varphi(\rho)=\frac{\gamma}{\gamma-1} \rho^{\gamma-1},\quad \gamma>1,}
which corresponds to the barotropic pressure law $\rho^\gamma$. The interaction potential $W:\R \to \R$ is assumed to be bounded from below---the precise assumptions on its regularity and growth will be given in Assumption~$(\mathcal{A}_W)$ in the next section.
 
 The second term on the right-hand side of  \eqref{main_eq2} is the viscosity dependent stress tensor with the viscosity coefficient $\mu=\mu(\rho)$. In our paper, we consider $\mu$ and $\varphi$ linked by the algebraic relation
\equ{\label{1.4}
    \mu(\rho) \coloneq \rho^2 \varphi'(\rho)=\gamma\rho^\gamma.
}
In particular, $\mu(\rho)=0$ when $\rho=0$. The third term is the nonlinear and nonlocal velocity alignment force with communication weight $\phi:\R \to \R_+$. Here $[\cdot,\cdot]$ stands for the commutator operator that can be equivalently expressed as
 \[
    \rho [\phi *, u]\rho \coloneq \rho (\phi*(u\rho) - u \phi*\rho) =   \rho(x) \int_\R \phi(x-y)(u(y) - u(x))\rho(y)\d y.
 \]
 Typically, $\phi$ is chosen as a radially decreasing function so that closer particles have a stronger tendency to be aligned than particles further away. Finally, the last term in the momentum equation \eqref{main_eq2} is the linear velocity damping with a parameter $\tau \geq 0$.

Simply put, the construction of weak solutions (cf.\ Definition~\ref{Def:1}) to the  system~\eqref{main_eq} is obtained as an accumulation point of a sequence of solutions $(\rho^\ep,u^\ep)$ to an approximate system with parameter $\ep>0$ (cf.\ Section~\ref{Sec:existence}), which hinges on providing relevant a-priori estimates. In this work, this is achieved by establishing a functional inequality for an appropriate linear combination of several energy functionals with 
one of them being the \emph{total energy}
\[
	\calE(\rho,u) \coloneq  \frac12\intO{\rho u^2} + \calF(\rho).
\]
Along sufficiently regular solutions of \eqref{main_eq}, the total energy is known to dissipate in time, i.e.
\[
    -\Dt\calE(\rho,u)
		= \int_\R \mu(\rho)|\pa_x u|^2\,\dx + \tau\intO{\rho u^2} + \calK_\phi(\rho,u),
\]
where the right-hand side is nonnegative and
\[
    \calK_\phi(\rho,u) \coloneq \frac12\iintO{ \phi(x-y)|u(x) - u(y)|^2 \rho(x)\rho(y)}.
\]
However, the control provided by the total energy alone does not provide sufficient control to characterize the accumulation points of the approximating sequence $(\rho^\ep,u^\ep)$. For this reason, we provide additional control of the so-called \emph{modulated energy}
\[
    \Emod(\rho,u) \coloneq \frac12\intO{  \rho |u + \pa_x \delta \calF(\rho)|^2 }+ \calF(\rho),
\]
which is a Bresch-Dejardins (BD) type energy \cite{BD03, BD04, BD07} that incorporates modulation of the nonlocal forces. Here, the quantity $w = u + \pa_x \delta \calF(\rho)$ can be seen as an augmented velocity.

In Section~\ref{Sec:apriori}, we formally show that for sufficiently smooth solutions to the Navier-Stokes system \eqref{main_eq}, the functional
\begin{align}\nonumber
    -\infty<\calJ_{\tau,\lambda}(\rho,u) \coloneq \Emod(\rho,u) + \tau\calF(\rho) + \lambda\calE(\rho,u),
\end{align}
for sufficiently large $\lambda\gg 1$,
satisfies the BD-type differential inequality
\begin{align}\label{eq:diff-inq}
    -\Dt\calJ_{\tau,\lambda}(\rho,u) \ge  \calD_{\tau,\lambda}(\rho,u),
\end{align}
where the dissipation---the right-hand side---is bounded from below by
\begin{align}\nonumber
    \calD_{\tau,\lambda}(\rho,u) \coloneq \frac12\intO{ \rho |\pa_x \delta \calF(\rho)|^2} + \ell_{\tau,\lambda}\intO{\rho|u|^2} + \lambda \intO{ \mu(\rho)|\pa_x u|^2} + c_\lambda\calK_\phi(\rho,u),
\end{align}
with appropriate constants $-\infty <\underline{\ell}\le \ell_{\tau,\lambda} $ and $c_\lambda>0$.

As it turns out, solutions to the approximate system satisfy a similar differential inequality (cf.\ Section~\ref{Sec:existence}), allowing us to deduce $\ep$-independent bounds in the form
\[
    \sup_{\ep>0}\left\{\sup_{t\in[0,T]} \calJ_{\tau,\lambda,\ep}(\rho^\ep(t),u^\ep(t)) + \int_0^T \calD_{\tau,\lambda,\ep}(\rho^\ep,u^\ep)\, \dt\right\} <+\infty.
\]
Together with the Mellet-Vasseur type estimate \cite{MV07} derived in Section~\ref{sec:compactness}, the uniform-in-$\ep$ bound above provides sufficient control to (a) deduce the integrated version of \eqref{eq:diff-inq}, i.e.
\begin{align}\label{eq:int-inq}
    \calJ_{\tau,\lambda}(\rho(t),u(t)) + \int_0^T \calD_{\tau,\lambda}(\rho,u)\, \dt \le \calJ_{\tau,\lambda}(\rho_0,u_0),
\end{align}
for the accumulation points $(\rho,u)$ of $(\rho^\ep,u^\ep)$, and (b) characterize these accumulation points as weak solutions of the Navier-Stokes system \eqref{main_eq} according to Definition~\ref{Def:1} below (cf.\ Theorem~\ref{Th:1}).

In the scenario when either $\tau>0$ or $\phi$ is uniformly bounded from below, i.e.\ $\phi\ge \underline{\phi}$ for some $\underline{\phi}>0$, the constant $\ell_{\tau,\lambda}$ can be chosen to be stricly positive. In particular, the dissipation $\calD_{\tau,\lambda}\ge 0$ is non-negative, and the BD-type integral inequality \eqref{eq:int-inq} can be used to deduce the long-time behaviour of the accumulation points $(\rho,u)$. Explicitly, we show that
\[
    \intO{ \rho\, u^2} \to 0 \qquad \text{and} \qquad  \intO{\rho \,(\pa_x \delta \calF(\rho))^2} \to 0\qquad\text{as $t\to \infty$}.
\]
If $W$ is additionally bounded and Lipschitz continuous, we further prove the existence of some $\rho_\infty\in H^1(\R)$ (cf.\ Theorem~\ref{Th:2}) such that for some (not relabelled) subsequence
\[
    \rho(t)\to\rho_\infty\quad  {\rm strongly\  in } \ L^{p}(\R), \ p\in[1,\infty),
\]
where the limit point $\rho_\infty$ satisfies
\[
    \int_\R \rho_\infty(\px \delta\calF(\rho_\infty))^2 \,\dx = 0\quad\Longleftrightarrow\quad \px\varphi(\rho_\infty) + \px W\ast \rho_\infty = 0\quad\text{on $\text{supp}(\rho_\infty)$}.
\]
In particular, this implies that the limit profile $\rho_\infty$ is a minimizer of the driving functional $\calF$.

\subsubsection*{Relation to previous work}

The pressureless Euler system, incorporating nonlocal interactions modelled by the interaction potential $W$ and a linear damping term (corresponding to constant alignment $\phi=1$), has been previously examined in \cite{CCZ}. In that work, the authors established the global existence of classical solutions over time and analyzed their long-time behaviour, contingent upon initial data falling below a specific threshold.  
When the interaction potential is given as the Riesz one, the rigorous derivation of the pressureless Euler system from its microscopic description was investigated in \cite{Ser20}. We also refer to \cite{KT15} for the study of the long-time asymptotics of solutions to the Euler system with pressure using relative entropy methods and to \cite{CCT19} using second-order Wasserstein calculus.
 
In the case $W, \mu \equiv 0$, the system \eqref{main_eq} becomes the so-called {\it Euler-alignment system}, also known as a macroscopic description of the celebrated Cucker-Smale model \cite{CS1,CS2}. The global existence of regular solutions and their long-time asymptotics were well studied in \cite{CJ24, HT08, ST1, ST2, ST3} under the smallness and smoothness assumptions on the initial data. The critical threshold phenomena in the pressureless Euler-alignment model were analysed in \cite{TT, CCTT} and more recently in \cite{BLT}. The rigorous derivation from its kinetic formulation was discussed in \cite{CCJ21, FK19} and \cite{CH24, CK23} in the case of regular and singular communication weights, respectively. In a recent work \cite{CC21}, the rigorous derivation of the pressureless Euler-alignment model from the Cucker-Smale model by means of mean-field limits was established.  We refer to \cite{CCP, CHL17, MMPZ, Shv21, Shvpre} and references therein for a general survey and recent development of swarming models.

 An intriguing question arising from these studies concerns the existence and behaviour of solutions beyond the blow-up time.  For the Euler-alignment model, the existence and uniqueness of these solutions have been explored through the sticky particles approximation in \cite{LT}. On the other hand, for the pressureless Euler with nonlocal interaction forces, the equivalence between Lagrangian and entropy solutions existing globally in time was shown in the recent work of Carrillo and Galtung \cite{CG23}. 
The motivation for the present work was similar. However, instead of delving into the different notions of measure-valued solutions, we focus on the global-in-time existence of weak solutions and their long-time asymptotics for a certain approximation of the original system from \cite{CCZ}. The discussion of the approximation with viscosity and pressure, for instance, can be found in \cite{CWKZ2}, where a relative entropy inequality was derived akin to the earlier work of Haspot \cite{Haspot1D}.

Despite the imposition of the relation between the viscosity coefficient and adiabatic exponent of pressure as in \eqref{1.4},  the authors of \cite{CWKZ2}  were unable to derive global estimates over time.  The equation \eqref{1.4} proved crucial for conducting the pressureless and inviscid limit simultaneously, and to arrive at the Euler system from \cite{CCZ}. Performing these limit passages independently requires using compensated compactness techniques of Chen and Perepelitsa \cite{CP1,CP2} and is investigated in the forthcoming work \cite{CCYZ}.

The key observation of our paper is that a suitable global in-time estimate holds provided pressure and viscosity are related by \eqref{1.4} and that a similar relation is imposed between their non-local counterparts: the attraction-repulsion potential $W$ and the communication weight $\phi$. In a sense, we further explore the two-velocity structure of hydrodynamic models, inspired by Bresch and Desjardins \cite{BD03}, but our augmented velocity $w$ also encompasses the gradient of the nonlocal function of density, i.e. $w=u+ \pa_x \delta \calF(\rho)$.
The analogy between nonlocal alignment and viscous dissipation was recognized and discussed earlier, as detailed in \cite[Chapter 5]{MMPZ}. We also draw attention to the works of Do et al.  \cite{Detal} and the series by Shvydkoy and Tadmor \cite{ST1, ST2, ST3}, with various results on fractional Euler alignment system in the 1D torus, exploiting the extra conservation law for the so-called {\it active potential}, see also \cite{CDNP20, CDS20}. A similar approach involving the utilization of an additional conservation law for artificial momentum was applied by Haspot and Zatorska in  \cite{HaZa} to perform the pressureless limit of the Navier-Stokes system, and by Constantin et.al. \cite{CDNP20, CDS20} to prove the existence of solutions for a general class of density-degenerate viscous models.

It is important to note that all the results mentioned above pertain to the one-dimensional spatial domain. For insights into the long-term behaviour of solutions in a general multi-dimensional hydrodynamic model of collective motion (with constant viscosity), we refer to \cite{CWKZ1}. 
Additionally, Carrillo and Shu recently investigated the pressureless Euler-Poisson system with quadratic confinement in a spherically symmetric multi-dimensional context \cite{CS23}; we refer to this paper for an up-to-date overview of results on that system in the context of continuous collective behaviour models. 

The technique of proving the existence of solutions for our system is hugely inspired by the papers \cite{HaZa, Jiu, LLX}, which in turn rely on the a-priori estimates derived in the earlier works of Bresch and Desjardins \cite{BD03, BD04, BD07} and of Mellet and Vasseur \cite{MV07}. The existence of solutions to the multi-dimensional compressible Navier-Stokes equations with density-dependent viscosities satisfying these a-priori estimates was proven in the seminal work of Vasseur and Yu \cite{VY16}, see also \cite{LX15}. The construction of the approximate solution in a one-dimensional setting is considerably less complex: we adapt the iterative scheme developed by Constantin et al.  \cite{CDS20}, along with the construction of approximate initial data by Chen et al.  \cite{CHWY2023} to the nonlocal case.

\subsubsection*{Outline of the paper} Section \ref{sec:main} is devoted to stating our main results. We then present the main a-priori estimates for regular solutions of system \eqref{main_eq} in Section \ref{Sec:apriori}. In particular, we show that the modulated energy estimate is satisfied uniformly in time. Then, in Section \ref{Sec:existence}, we outline the existence proof of regular approximate solutions, with some details postponed to Appendix \ref{AppendixA} and \ref{app:initial}. Two subsequent sections are devoted to the recovery of the original system and the limit passage in the energy inequalities. Long-time asymptotics of solutions stated in our main result Theorem~\ref{Th:2} is proven in Section \ref{Sec:lt}.

\section{Main results}\label{sec:main}
 We supplement system \eqref{main_eq} with the initial data
	\equ{\label{initiald}
		(\vr(t,\cdot), u(t,\cdot))|_{t=0} = ( \vr_0, u_0),}
	for which we assume that
	\begin{equation}
		\begin{gathered}
			\vr_0\geq 0, \quad \vr_0\in L^1(\R)\cap L^\infty(\R), \quad 
			\partial_x(\vr_0^{\gamma-\frac{1}{2}})\in L^2(\R),\quad \vr_0 u_0^2\in L^1(\R)  \\
			|x|^{\kappa+2}\vr_0\in L^1(\R),\quad \vr_0|u_0|^{2+\kappa}\in L^1(\R),\qquad 0<\kappa\leq\min\left\{2\gamma-1,\frac2\gamma\right\}.
		\end{gathered}
		\label{ini}
	\end{equation}
 Since the mass of $\rho_0$ will be preserved in time, we can consider w.l.o.g.\ that $\intO{\vr_0(x)}=1$.

 Throughout this manuscript, we assume that the interaction potential $W$ and the communication weight $\phi$ satisfy the following conditions $(\mathcal{A})$:
\begin{itemize}
\item[$(\mathcal{A}_W)$] $W$ is bounded from below by some constant, and further it has a form of 
\[
W(x) = - |x| + \tw(x),
\]
where $\tw: \R \to \R$ is symmetric and satisfies
\begin{itemize}
\item[(i)] $\tw \in C(\R) \cap C^2(\R \backslash \{0\})$,  
\item[(ii)] $|\partial_x \tw(x)|\leq l_W(1+|x|)$ for all $x \in \R$ for some constant $l_W>0$, and
\item[(iii)] there exists some positive constants $c_{\phi,W}$ and $c_W$ such that
\[
(\pa_{xx}\tw)^- \leq c_{\phi, W}\phi\quad \mbox{for } x\neq 0\quad \mbox{and} \quad 
  (\pa_{xx}\tw)^+ \leq c_W \quad \mbox{for all } x \in \R,
\]
where $(f)^\pm=\max\{0,\pm f\}$ denote the positive and negative parts of a given function $f$.
\end{itemize}
\smallskip
\item[$(\mathcal{A}_\phi)$] $\phi \in C(\R \backslash\{0\})$ and for any $R>0$ 
\[
\phi \mathbf{1}_{B(0,R)} \in L^{\frac\gamma{\gamma-1}}(\R) \quad \mbox{and} \quad \phi \mathbf{1}_{\R \backslash B(0,R)} \in L^\infty(\R).
\]
\end{itemize}

\begin{remark} 
\begin{enumerate}
    \item[(a)] The interaction potential $W$ given as a sum of the Coulomb potential and quadratic confinement, i.e.
	\equ{\nonumber
		W(x)= - |x|+\frac{|x|^2}{2},}
clear satisfies assumption $(\mathcal{A}_W)$. 
    \item[(b)] If the communication weight $\phi$ and the interaction potential $W$ are given as 
    \[
    \phi(x) = \frac1{|x|^a} \quad \mbox{with}\;\;  a < \frac{\gamma-1}\gamma,\quad W = - |x| - |x|^{2-a},
    \]
    then they satisfy conditions $(\mathcal{A})$. Indeed, in this case,
    \[
        (\pa_{xx} \tw)^- = (2-a)(1-a) |x|^{-a} = (2-a)(1-a) \phi(x),\qquad (\pa_{xx} \tw)^+ = 0.
    \]
    \item[(c)] We deduce from Assumption ($\calA2$) the following lower bound on the interaction energy 
    \[
        \intO{\rho\,W\ast\rho} \ge c_0 - c_1\intO{|x|^2\rho},
    \]
    for some constants $c_0,c_1>0$, depending only on $\tw(0)$ and $l_W$.
\end{enumerate}
\end{remark}

Let us now define the notion of solutions to the system \eqref{main_eq}

 	\begin{definition}\label{Def:1} Let $T>0$ be fixed but arbitrary.
		The pair of functions $(\vr,\sqrt{\vr}u)$ is called a global-in-time weak solution to system \eqref{main_eq} with initial data \eqref{initiald} satisfying \eqref{ini} if
		\begin{equation}\nonumber 
			\begin{gathered}
				\vr\in C_{w} ([0,T]; L^\gamma(\R)), \qquad \vr\in L^\infty(0,T; L^1\cap L^\infty(\R)), \\
				\sqrt{\vr} u \in L^\infty(0,T;L^2(\R)),\qquad \partial_x (\vr^{\gamma-\frac{1}{2}})\in L^\infty(0,T; L^2(\R)).
			\end{gathered}       
		\end{equation}
		Moreover, the following weak formulations are satisfied.
  \begin{itemize}
    \item[(i)] The continuity equation holds 
		\begin{equation}
			\int_{\R}\vr\psi(t_2)\,\dx-\int_{\R}\vr\psi(t_1)\,\dx=
			\int^{t_2}_{t_1}\!\!\!\int_{\R}(\vr\partial_t\psi+\vr u\partial_x\psi)\,\dxdt
			\label{2.3}
		\end{equation}
		for any $0\leq t_1\leq t_2\leq T$ and any $\psi\in C^1_c( [0,T]\times \R)$.
    \item[(ii)] The momentum equation holds
		\equ{
&\int_{\R}\vr_0u_0\psi(0)\,\dx+\intTO{\lr{\vr u\pt\psi+\vr u^2\px\psi}}{-\langle\mu(\vr)\px u,\px\psi\rangle}\\
			&\hspace{6em}=\intTO{\rho \pa_x \delta\calF(\rho) \psi}-\intTO{\rho [\phi *, u]\rho\psi}+\tau\intTO{\rho u\psi},
			\label{2.4}
		}
		for any $\psi\in C^{\infty}_{c}([0,T)\times\R)$, where the diffusion term is defined as follows:
		\equ{\nonumber
			&\langle\mu(\vr)\px u,\px\psi\rangle\\
			&\hspace{3em}=-\intTO{\frac{\mu(\vr)}{\sqrt{\vr}}\sqrt{\vr}u\,{\partial_{xx}}\psi}-
			\frac{2\gamma}{2\gamma-1}\intTO{\px\lr{\frac{\mu(\vr)}{\sqrt{\vr}}}\sqrt{\vr} u\,{\px\psi}}.	
		}
  \end{itemize}
\end{definition}

\medskip

Before stating our main results, let us recall the definition of the total energy $\calE$ and the modulated energy $\Emod$ for sufficiently smooth pairs $(\rho,u)$ given by,
	\begin{align*}
	\calE(\rho,u) &=  \frac12\intO{\rho u^2} + \calF(\rho),\\
	\Emod(\rho,u) &= \frac12\intO{  \rho |u + \pa_x \delta \calF(\rho)|^2 }+ \calF(\rho).
	\end{align*}
We further recall the functionals 
\begin{align*}
    \calJ_{\tau,\lambda}(\rho,u) &= \Emod(\rho,u) + \tau\calF(\rho) + \lambda\calE(\rho,u),\\
    \calD_{\tau,\lambda}(\rho,u) &= \frac12\intO{ \rho |\pa_x \delta \calF(\rho)|^2} + \ell_{\tau,\lambda}\intO{\rho|u|^2} + \lambda \intO{ \mu(\rho)|\pa_x u|^2} + c_\lambda\calK_\phi(\rho,u).
\end{align*}
Our first result pertains to the existence of weak solutions to the nonlocal Navier-Stokes system \eqref{main_eq}.




\begin{theorem}\label{Th:1} Let $\tau\ge 0$ and the conditions $(\mathcal{A})$ hold. Then there exist a weak solution to system \eqref{main_eq} in the sense of Definition \ref{Def:1}, and $\lambda>0$ such that the modulated energy inequality
 \equ{\label{mod_int}
   \sup_{t\in[0,T]} \calJ_{\tau,\lambda}(\rho,u)(t)+\int_0^{T} \calD_{\tau,\lambda}(\rho,u)\,\dt\leq \calJ_{\tau,\lambda}(\rho_0,u_0)
}
holds with constants $\ell_{\tau,\lambda}\in\R$ and $c_\lambda > 0$.

If either $\tau>0$ or $\phi$ is uniformly bounded from below, i.e.\ $\phi\ge \underline{\phi}$ for some constant $\underline{\phi}>0$, then $\lambda>0$ can be chosen such that $\ell_{\tau,\lambda}>0$. In these cases, $\calD_{\tau,\lambda}\ge 0$, and the modulated energy inequality \eqref{mod_int} is satisfied uniformly w.r.t.\ $T$.
\end{theorem}

 \begin{remark}
 The hypothesis \eqref{ini} implies 
     \[
	\calE(\rho_0,u_0) \coloneq  \frac12\intO{\rho_0 u_0^2}  + \calF(\rho_0) < \infty ,
	\]
	\[
	 \Emod(\rho_0,u_0) \coloneq \frac12\intO{  \rho_0 |u + \pa_x \delta \calF(\rho_0)|^2} + \calF(\rho_0) <\infty .
	\]
 \end{remark}
	Once the existence of solutions is established we turn our attention to their long-time asymptotic. Our second main result reads.

		\begin{theorem}\label{Th:2} Let $(\rho,\sqrt{\vr}u)$ be the global-in-time weak solution to the Navier-Stokes system \eqref{main_eq} provided by Theorem \ref{Th:1}, and let either $\tau>0$ or $\phi$ uniformly bounded from below. Then, we have
			\[
			\intO{ \rho\, u^2} \to 0 \qquad \text{and} \qquad  \intO{\rho \,(\pa_x \delta \calF(\rho))^2} \to 0\qquad\text{as $t\to \infty$}.
			\]
Moreover, if $W$ is bounded and Lipschitz continuous, then there exists $\rho_\infty\in H^1(\R)$ such that
   \[
    \begin{gathered}
   \rho(t)\to\rho_\infty\quad  strongly\ in \ L^{p}(\R), \ p\in[1,\infty),\\ 
   \sqrt{\rho(t)}\, u(t)\to 0 \quad strongly\  in \ L^1(\R).
   \end{gathered}
   \]
   up to a subsequence, and $\rho_\infty$ is characterized by
			\[
			 \px\varphi(\rho_\infty) + \px W * \rho_\infty =0 \quad \mbox{on } {\rm supp}(\rho_\infty). 
			\]
		\end{theorem}
		
		\begin{remark} 
			The boundedness and Lipschitz continuity conditions on the interaction potential $W$ can be relaxed to requiring 
			\bq\label{add_con}
			\sup_{t \geq 0} \lt( \intO{\rho|W * \rho|} + \intO{\rho|\pa_x W * \rho|^2}\rt) < \infty.
			\eq
			Clearly, if the interaction potential $W$ is Lipschitz and bounded, then \eqref{add_con} holds.
		\end{remark}
	

	\section{A priori energy and modulated energy estimate}\label{Sec:apriori}
	This section is dedicated to the derivation of the modulated energy inequality
 that allows us to deduce the following  a-priori estimates:
\equ{\label{est_mod}
	&\sup_{t\in[0,T]}\lr{\intO{\vr|u|^2}+\calF(\rho)}\\
	&\qquad+\intTO{\mu(\vr)|\px u|^2}+C(\tau)\intTO{\rho u^2}+\int_0^T\calK_\phi(\rho,u)\,\dt\leq C(\rho_0,u_0),\\
	&\sup_{t\in[0,T]}\intO{\vr|\px\delta\calF(\rho)|^2}+\intTO{ \rho (\pa_x \delta \calF(\rho))^2}\leq C(\rho_0,u_0),
}
for some constant $C(\tau)>0$.

\medskip

We have the following result.
	
	\begin{lemma}\label{lem:bd-estimate}
		Let $(\vr,u)$ be a regular solution to system \eqref{main_eq} for $\tau\geq0$ on the interval $[0,T]$, and let Assumption~$(\mathcal{A})$ be satisfied. Then there exists $\lambda > 0$  such that
    \equ{\label{est:BD}
        \Dt\calJ_{\tau,\lambda}(\rho,u) \le -\calD_{\tau,\lambda}(\rho,u)
    }
		with constants $\ell_{\tau,\lambda}\in\R$ depending on $\tau$ and $\lambda$, $c_\lambda > 0$ depending only on $\lambda$.
	\end{lemma}
 
Before proving this lemma let us make a few observations.
\begin{remark}\nonumber 
\begin{itemize}
    \item[(a)] If $\tau>0$, then $\lambda$ can be chosen sufficiently large such that $\ell_{\tau,\lambda}>0$. In particular, the right-hand side of \eqref{est:BD} is nonpositive, and therefore the estimates in \eqref{est_mod} are uniform in time. The same conclusion holds true if instead of $\tau>0$, we assume that $\phi$ is uniformly bounded from below, i.e.\ $\phi\ge \underline{\phi}$ for some constant $\underline{\phi}>0$.
    \item[(b)] If $\tau=0$ and $\phi$ is allowed to vanish, then $\ell_{\tau,\lambda}=-c_W^2<0$, and applying Gr\"{o}wnwall's lemma to \eqref{est:BD} provides the upper bound estimates in \eqref{est_mod} depending on $T>0$, which can be arbitrary large but finite.
\end{itemize}
\end{remark}

\begin{proof}[Proof of Lemma~\ref{lem:bd-estimate}]
    We only prove the case $\tau>0$ as the case when $\phi\ge \underline{\phi}$ can be done analogously.
    
		Straightforward computation gives that the mass of $\rho$ is preserved in time, i.e.
  \[
    \intO{\vr(t,x)} = \intO{\vr_0(x)} = \frkm,
    \] 
    and that
		\bq\label{energy0}
		\Dt\calE(\rho,u)
		= - \int_\R \mu(\rho)|\pa_x u|^2\,\dx-\tau\intO{\rho u^2} - \frac12\calK_\phi(\rho,u).
		\eq
Since $W$ is bounded from below, we deduce from the above that
  \[
  \sup_{t\in[0,T]} \|\rho(t)\|_{L^1 \cap L^\gamma} < +\infty.
  \]
  In particular, this implies
  \[
    \begin{aligned}
     \sup_{t\in[0,T]}\|\phi * \rho(t)\|_{L^\infty} &\leq \sup_{t\in[0,T]}\|\phi \mathbf{1}_{|\cdot| \leq 1} * \rho(t)\|_{L^\infty} + \sup_{t\in[0,T]}\|\phi \mathbf{1}_{|\cdot| \geq 1} * \rho(t)\|_{L^\infty} \\
     &\leq \sup_{t\in[0,T]}\|\phi \mathbf{1}_{|\cdot| \leq 1}\|_{L^{\frac\gamma{\gamma-1}}}\|\rho(t)\|_{L^\gamma} + \sup_{t\in[0,T]}\|\phi \mathbf{1}_{|\cdot| \geq 1}\|_{L^\infty}\|\rho(t)\|_{L^1} <+\infty.
    \end{aligned}
  \]
  
	For the modulated energy, we estimate
		\equ{\label{split_en}
			\Dt\Emod(\rho,u) &= \Dt\calE(\rho,u) + \frac12\Dt\intO{ \rho |\pa_x \delta \calF(\rho)|^2} + \Dt\intO{ \rho u \pa_x \delta \calF(\rho)}\\
			&=: I_1 + I_2 + I_3.
		}
	We begin by rewriting $I_2$ using the form of $\delta \calF$ in \eqref{dF} and the continuity equation \eqref{main_eq1} to obtain
	\begin{align}\nonumber
        \begin{aligned}
			&\intO{\rho \pa_x \delta \calF(\rho)\, \pa_t \pa_x \delta \calF(\rho)} \cr
			&\hspace{4em} = - \intO{\rho \pa_x \delta \calF(\rho)\, \pa_x \lt( (\pa_x \varphi) u + \varphi'(\rho) \rho \pa_x u + \pa_x W * (\rho u) \rt)}\cr
			&\hspace{4em} = - \intO{\rho \pa_x \delta \calF(\rho)  \lt((\pa_{xx} \varphi) u + (\pa_x \varphi)\pa_x u + \pa_x( \varphi'(\rho) \rho \pa_x u) + \pa_{xx} W * (\rho u) \rt)}\cr
			&\hspace{4em} = \frac12 \intO{\pa_x(\rho u) |\pa_x \delta \calF(\rho)|^2 - \int \rho [\pa_{xx}W *, u]\rho \pa_x \delta \calF(\rho)}\cr
			&\hspace{4em}\qquad - \intO{\rho\lt( \pa_x( \varphi'(\rho) \rho \pa_x u)+ (\pa_x u) (\pa_x \varphi)\rt) \pa_x \delta \calF(\rho)}.
        \end{aligned}
	\end{align}
		Since $\mu(\rho) = \rho^2 \varphi'(\rho)$, we can express the integrand in the last term on the right-hand side as
		\[
		\rho\lt(\pa_x( \varphi'(\rho) \rho \pa_x u)+ (\pa_x u) (\pa_x \varphi)\rt) = \pa_x (\mu(\rho)\pa_x u),
		\]
		and thus
		\equ{\nonumber
			\intO{\rho \pa_x \delta \calF(\rho) \,\pa_t \pa_x \delta \calF(\rho)} 
			&= \frac12 \intO{\pa_x(\rho u) (\pa_x \delta \calF(\rho))^2} \\
   &\quad - \intO{\rho [\pa_{xx}W *, u]\rho \pa_x \delta \calF(\rho)}- \intO{\pa_x (\mu(\rho)\pa_x u) \pa_x \delta \calF(\rho)}.
		}
		This yields
		\begin{align*}
			I_2=- \intO{\rho [\pa_{xx}W *, u]\rho \pa_x \delta \calF(\rho)}- \intO{\pa_x (\mu(\rho)\pa_x u) \pa_x \delta \calF(\rho)}.
		\end{align*}
		For $I_3$, we obtain
		\begin{align}\label{BDb}
            \begin{split}
			I_3 &= -\intO{\pa_x (\rho u^2) \pa_x \delta \calF(\rho)} - \intO{\rho |\pa_x \delta \calF(\rho)|^2} + \intO{\rho [\phi*, u]\rho \pa_x \delta \calF(\rho)} \\
			&\qquad + \intO{\pa_x (\mu(\rho)\pa_x u) \pa_x \delta \calF(\rho)} + \intO{\rho u \pa_t \pa_x \delta \calF(\rho)}-\tau\intO{\rho u \pa_x \delta \calF(\rho)}.
            \end{split}
		\end{align}
		Similarly, as before, we can find 
		\begin{align}
			\intO{\rho u \pa_t \pa_x \delta \calF(\rho)} &= \intO{\pa_x (\rho u^2) \pa_x \delta \calF(\rho)} - \intO{\rho u[\pa_{xx}W *, u]\rho} \nonumber \\
			&\qquad- \intO{\rho u \lt((\pa_x u)(\pa_x \varphi) + \pa_x (\varphi'(\rho) \rho \pa_x u) \rt)}\cr
			&=\intO{\pa_x (\rho u^2) \pa_x \delta \calF(\rho)} - \intO{\rho u[\pa_{xx}W *, u]\rho} + \intO{\mu(\rho)|\pa_x u|^2}\nonumber
		\end{align}
and integrating by parts in the last term of \eqref{BDb}, and using the continuity equation, we obtain
	\equ{\nonumber
	-\tau\intO{\rho u\px\delta\calF(\rho)}
 &= \tau\intO{\px(\rho u)\delta\calF(\rho)}=-\tau\intO{\pt\rho\,\delta\calF(\rho)}
 = -\tau\Dt\calF(\rho).
 }
		This gives
		\begin{align*}
			I_3 &=  - \intO{\rho |\pa_x \delta \calF(\rho)|^2} + \intO{ \rho [\phi*, u]\rho \pa_x \delta \calF(\rho)}   -\tau\Dt\calF(\rho) \cr
			&\qquad + \intO{\pa_x (\mu(\rho)\pa_x u) \pa_x \delta \calF(\rho)} - \intO{\rho u[\pa_{xx}W *, u]\rho} + \intO{\mu(\rho)|\pa_x u|^2}.
		\end{align*}
		Putting $I_2$ and $I_3$ together, we arrive at
		\equ{\nonumber
			I_2 + I_3 =  &- \intO{\rho |\pa_x \delta \calF(\rho)|^2}+ \intO{\rho [(\phi - \pa_{xx}W)*, u]\rho \pa_x \delta \calF(\rho)} \\ 
			&\qquad - \intO{\rho u[\pa_{xx}W *, u]\rho}+ \intO{\mu(\rho)|\pa_x u|^2}  -\tau\Dt\calF(\rho).
		}
		We then combine the above with \eqref{energy0} to express \eqref{split_en} as:
		\begin{align}\nonumber
			\begin{aligned}
				 \Dt \lt(\Emod(\rho,u) +\tau\calF(\rho)\rt) &= - \intO{\rho |\pa_x \delta \calF(\rho)|^2}+ \intO{\rho [(\phi - \pa_{xx}W)*, u]\rho \pa_x \delta \calF(\rho)} \cr
				&\quad - \frac12\iintO{(\phi + \pa_{xx}W)(x-y)(u(x) - u(y))^2 \rho(x)\rho(y)}\cr
				&\quad -\tau\intO{\rho u^2}\cr
  &   =: J_1 + J_2 + J_3 + J_4.
			\end{aligned}
		\end{align}
Note that 
\[
J_2 = \intO{\rho [(\phi - \pa_{xx}\tw)*, u]\rho \pa_x \delta \calF(\rho)} ,
\]
and
\[
J_3 = - \frac12\iintO{(\phi + \pa_{xx}\tw)(x-y)(u(x) - u(y))^2 \rho(x)\rho(y)}.
\]
Since 
\[
|\phi - \pa_{xx}\tw| \leq  (1 + c_{\phi, W})\phi + c_W,
\]
 we obtain
	\begin{align*}
		| J_2 | &\leq c_W \iintO{ |u(x) - u(y)| |\pa_x \delta \calF(\rho)(x)| \rho(x)\rho(y)   } \\
		&\quad+(1+c_{\phi,W} )  \iintO{ \phi(x-y)|u(x) - u(y)| |\pa_x \delta \calF(\rho)(x)| \rho(x)\rho(y) } =:J_{2,1} + J_{2,2}.
	\end{align*}
 	We then estimate the term $J_{2,1} $ as
		\begin{align*}
		J_{2,1} &\leq  c_W ^2   \iintO{(u(x) - u(y))^2 \rho(x)\rho(y)}  + \frac{1}{4}  \intO{\rho |\pa_x \delta \calF(\rho)|^2}\\
  &\leq2 c_W ^2   \intO{ \rho u^2}  + \frac{1}{4}  \intO{\rho |\pa_x \delta \calF(\rho)|^2}.
		\end{align*}
	Next, we estimate $J_{2,2}$ to have
	\begin{align*}
		&\iintO{ \phi(x-y)|u(x) - u(y)| |\pa_x \delta \calF(\rho)(x)| \rho(x)\rho(y) }  \\
		&\qquad= \iintO{ \sqrt{\phi(x-y)}|u(x) - u(y)| \sqrt{\rho(x)\rho(y)}   \sqrt{ \phi(x-y) \rho(y)} \sqrt{\rho(x)}|\pa_x \delta \calF(\rho)(x)|  }  \\
		&\qquad\leq \Vert \phi \ast \rho \Vert_{L^\infty} \calK_{\phi}(\vr,u) + \frac{1}{4}\intO{\rho |\pa_x \delta \calF(\rho)|^2}.
	\end{align*}
Together, this gives
\[
J_2 \leq 2c_W^2  \intO{ \vr \vert u \vert^2 } +  \frac12   \intO{\rho |\pa_x \delta \calF(\rho)|^2} + (1+c_{\phi,W} )^2 \Vert \phi \ast \rho \Vert_{L^\infty}  \calK_{\phi}(\vr,u).
\]
For the term $J_3$, we observe that the part corresponding to $\varphi +(\pa_{xx}\tw)^+$ is nonpositive. Hence, we only need to control the contribution due to $(\pa_{xx}\tw)^-$;
we use \[
-c_{\phi, W} \phi \leq -\bigl(\pa_{xx} \tw\bigr)^{-} \leq	\phi + \pa_{xx}\tw
\]
to obtain
\[
J_3 \leq \frac12 c_{\phi, W} \calK_{\phi}(\vr,u).
\] 
Combining the estimate above yields
\begin{align*}
	\Dt \lt(\Emod(\rho,u) +\tau\calF(\rho)\rt)
 &\leq - \frac12\intO{\rho |\pa_x \delta \calF(\rho)|^2} -(\tau- 2c_W^2)\intO{\rho u^2} \cr
 &\hspace{6em}+ \lt(\frac12 c_{\phi, W} + (1 + c_{\phi, W})^2  \|\phi * \rho\|_{L^\infty}\rt)\calK_\phi(\rho,u).
	\end{align*}
	Consequently,
	\begin{align*}
		\Dt \lt(\Emod +\tau\calF(\rho) + \lambda \calE \rt) 
  &\leq - \frac12\intO{\rho |\pa_x \delta \calF(\rho)|^2}  - \ell_{\tau,\lambda}\intO{\rho u^2}  \cr
  &\hspace{6em} -\lambda \intO{\mu(\rho)|\pa_x u|^2}- c_\lambda  \calK_\phi(\rho,u).
	\end{align*}
	We finally choose $\lambda > 0$ large enough so that 
	\[
	   \ell_{\tau,\lambda}\coloneq(1+\lambda)\tau-2c_W^2>0,\qquad c_\lambda\coloneq \frac\lambda 2- \lt(\frac12 c_{\phi, W} + (1 + c_{\phi, W})^2\|\phi * \rho\|_{L^\infty}\rt) > 0,
	\]
  to conclude the desired result. 
\end{proof}

\section{Existence of regular solutions to the approximate system} \label{Sec:existence}

The goal of this section is to discuss the existence and uniqueness of an approximate system and to deduce the functional inequalities seen in Section~\ref{Sec:apriori} for the solution of the approximate system. 

\subsection{The approximate system} We approximate our system on the whole space $\R$ by a family of problems on a bounded domain $\Omega_\ep\coloneq(-R_\ep,R_\ep)\subset\R$, $\ep>0$, $R_\ep>0$ and look for regular solutions $(\rho^{\ep},u^{\ep})$ to the approximate system, where the dependence of $R_\ep>0$ on $\ep>0$ will be clarified later.
\begin{subequations}\label{app_eq}
	\begin{align}
		\pa_t \rho + \pa_x (\rho u) &= 0, \qquad (t,x) \in (0,T)\times\Om_\ep,\label{app_eq1}\\
		\pa_t (\rho u) + \pa_x (\rho u^2) &= - \rho \pa_x \delta\calF_\ep(\rho) + \pa_x (\mu_\ep(\rho)\pa_x u) + \rho [\phi_\ep *^\ep, u]\rho -\tau\rho u, \label{app_eq2}
	\end{align}
\end{subequations}
where
\[
\begin{gathered}	\varphi_\ep(\rho)\coloneq\varphi(\rho)+\frac{\ep^\beta \alpha}{\alpha-1}\rho^{\alpha-1}, \quad \alpha\in(0,\sfrac12), \quad \beta= \frac{\gamma-\alpha}{\gamma-\sfrac12} +1 ,\quad \mu_\ep(\rho)\coloneq \rho^2\varphi_\ep'(\rho) = \mu(\vr)+\ep^\beta\alpha\rho^\alpha ,\\
	\calF_\ep(\rho) \coloneq \intM{\int_0^{\rho(x)}\varphi_\ep(r)\d r} + \frac12 \intM{\rho(x) W *^\ep \rho(x)},\qquad \delta \calF_\ep(\rho) = \varphi_\ep(\rho) + W *^\ep \rho,\\
	W*^\ep\rho \coloneq\int_{\Omega_\ep} W(\cdot-y)\rho(y)\, \dy,\qquad [\phi_\ep *^\ep, u]\rho \coloneq\int_{\Omega_\ep}\phi_\ep(\cdot -y)(u(y)-u(\cdot))\rho(y)\, \dy.
\end{gathered}
\]
Here, $\phi_\ep$ is a regularized symmetric communication weight satisfying
\[
     \phi_\ep \to \phi \quad\text{in\;\;$L^{\gamma^\prime}$},\quad \gamma^\prime = \frac{\gamma}{\gamma-1}.
\]
To simplify notation, we drop the upper-indexes $\ep$ in all places corresponding to the sequence and only keep it to denote the truncation of the convolution terms. Note that both $\rho$ and $u$ are only defined on $\Omega_\ep$ while the interaction potential $W$ and communication weight $\phi_\ep$ are defined on $\R$.

We further associate the following functionals to the regularized system:
\begin{align*}
\calJ_{\tau,\lambda,\ep}(\rho,u) &= \Emod_\e(\rho,u) + \tau\calF_\e(\rho) + \lambda\calE_\e(\rho,u),\\
\calD_{\tau,\lambda,\ep}(\rho,u) &= \frac12\intM{ \rho |\pa_x \delta \calF_\ep(\rho)|^2} + \ell_{\tau,\lambda}\intM{\rho|u|^2} + \lambda \intM{ \mu_\ep(\rho)|\pa_x u|^2} + c_\lambda\calK_{\phi^\ep}(\rho,u),
\end{align*}
with the driving functional
\[
\calF_\ep(\rho) = \frac1{\gamma-1}\intM{\rho^\gamma} + \frac12 \intM{ \rho\, W *^\e \rho},
\]
and the total and modulated energies
\begin{align*}
\calE_\e(\rho,u) =  \frac12\intM{\rho |u|^2}  + \calF_\ep(\rho),\qquad
\Emod_\e(\rho,u) = \frac12\intM{  \rho |u + \pa_x \delta \calF_\ep(\rho)|^2} + \calF_\e(\rho).
\end{align*}

\subsection{The approximate initial data}

Here, we discuss the approximation of the initial data and show the convergences of the approximated energy functionals. 

We first recall our main assumptions on the initial data: 	
\begin{align}\label{ini_app_app}
	\begin{gathered}
		\vr_0\geq 0, \quad \vr_0\in L^1(\R)\cap L^\infty(\R), \quad 
		\partial_x(\vr_0^{\gamma-\frac{1}{2}})\in L^2(\R),\quad \vr_0 u_0^2\in L^1(\R)  \\
		|x|^{\kappa+2}\vr_0\in L^1(\R),\quad \vr_0|u_0|^{2+\kappa}\in L^1(\R),\qquad 0<\kappa\leq\min\left\{2\gamma-1,\frac2\gamma\right\}.
	\end{gathered}
\end{align}
Motivated by \cite{CHWY2023},  we approximate our initial density $\rho_0$ by
  \begin{align}\label{id-d-1}
	\vr_0^{\ep} \coloneq \frac{\frkm}{Z_\ep} \mathbf{1}_{\overline\Omega_\ep} \widetilde{\vr}_0^{\ep},\qquad \widetilde{\vr}_0^{\ep}\coloneq \frac{\frkm}{Z_\ep}\mathbf{1}_{\overline\Omega_\ep}\lr{\rho_0^{\gmhalf} \ast \eta_{\ep^{\gmhalf}} +\ep e^{-x^2} }^{\frac{1}{\gmhalf}},
\end{align}
where $Z_\e$ is the normalizing constant, i.e.
\[
Z_\e = \int_\R  \widetilde{\vr}_0^{\ep} \mathbf{1}_{\overline\Omega_\ep} \d x,
\]
and $\Om_\e = (-R_\e, R_\e)$ with $0 < R_\e \nearrow +\infty$ as $\e \to 0$ which will be determined appropriately later. Note that $\widetilde{\vr}_0^\ep \in C^\infty(\R)$ and $\widetilde{\vr}_0^{\ep}  >0$ for all $x \in \R$.

For the approximation of the momentum $m_0 = \rho_0 u_0$, we introduce
\[
w_0 \coloneq \sqrt{\rho_0} u_0 \quad \mbox{and} \quad \tilde{w}_0 \coloneq \rho_0^{\frac{1}{2+\kappa}}. 
\]
Note that 
\[
    w_0 \in L^2(\R), \quad \tilde{w}_0 \in L^{2+\kappa}(\R), \quad \text{and} \quad m_0 \in L^{\frac{2\gamma}{\gamma+1}} (\R)
\]
due to \eqref{ini_app_app}. We next consider an approximation sequence $\tilde{w}_{0}^\ep \in C_c^\infty(\R)$ such that 
\begin{equation}\label{app_tw0}
\tilde{w}_0^\ep \rightarrow \tilde{w}_0 \text{ in } L^{2+\kappa}(\R) . 
\end{equation}
Using this newly defined approximation  $\tilde{w}_{0}^\e$, we define our approximated initial velocity by
\begin{equation}\label{id-v-1}
{u}_0^\ep \coloneq \frac{ \tilde{w}_0^\ep  }{ {(\rho_0^\ep)}^{\frac{1}{2+\kappa}} }\xi_\ep, 
\end{equation}
with $\xi_\e \in C^\infty(\R)$ satisfying
\[
{\bf 1}_{(\frac{-R_\e}{4}, \frac{R_\e}{4})} \leq \xi_\e \leq {\bf 1}_{(\frac{-R_\e}{2}, \frac{R_\e}{2})}.
\]
Note that $u^\e_0 \in C^\infty_c(\Omega_\e)$. Thus, the approximated initial velocity satisfies the following homogeneous Dirichlet boundary condition:
\equ{
	u=0\quad\text{on\;\;$\partial\Omega_\ep$}.\label{app:bc}
}
Naturally, the approximation of $w_0$ is given by
\begin{align*}
    w_0^\ep \coloneq  {(\rho_0^\ep)}^{\frac{1}{2}-\frac{1}{2+\kappa}} \tilde{w}_0^\ep \xi_\ep = \sqrt{\vr_0^\ep} u_0^\ep.
\end{align*}

For the sequence of approximations $(\rho_0^\ep,u_0^\ep)$ constructed above, we have the following result.

\begin{proposition}\label{prop_convJ} Let $R_\e > 0$ be given by
\begin{align}\label{choiceR_ep}
    R_\ep= a |\log \ep |^{\theta} \text{ with } 0<\theta <\min\left\{ \frac{1}{2}, \frac{3}{4(\gamma-1) \widetilde{\kappa}} \right\} \;\text{ and }  a = \sqrt{\frac{\gmhalf}{2(\gamma-\alpha)}} > 0,
\end{align}
for $\e < e^{-1}$, where $\widetilde{\kappa} \coloneq(2+\kappa) \frac{\gmhalf}{\gamma-1} +1$. Then we have
\[
\limsup_{\e \to 0}\int_{-R_\e}^{R_\e}{\vr_0^{\ep} |{u}_0^\ep|^{2+\kappa} } \d x \leq \intO{\vr_0 |{u}_0|^{2+\kappa} } \quad \mbox{and} \quad \calJ_{\tau,\lambda,\e}(\rho^\e_0,u^\e_0) \to \calJ_{\tau,\lambda}(\rho_0,u_0) \quad \mbox{as} \quad \e \to 0.
\]
\end{proposition}

The proof of Proposition \ref{prop_convJ} is rather technical and lengthy, thus for smoothness of reading, we postpone the details of proof to Appendix \ref{app:initial}.

\subsection{Formulation of the main result}

The goal of the rest of this section is to discuss the existence of a global-in-time regular solution to the approximate system \eqref{app_eq} for fixed $\ep>0$. 
We will frequently make use of the following notation for function spaces:
\begin{equation} \label{def:spaces}
\begin{aligned}
	& \mathcal{X}_k(T)\coloneq\{ f\in L^\infty(0,T;H^k(\Omega_\ep)): \; \pa_{t} f \in L^\infty(0,T;H^{k-1}(\Omega_\ep)\},\\
	& \mathcal{Y}_k(T)\coloneq\{ g \in L^2(0,T;H^{k+1}\cap H^1_0(\Omega_\ep) \cap L^\infty(0,T;H^k(\Omega_\ep) : \pa_{t} g \in L^2(0,T;H^{k-1}(\Omega_\ep) \}.
\end{aligned}
\end{equation}

The main result in this section is summarized in the following statement.

	\begin{theorem}\label{Th:3}
	Let $\ep>0$ be fixed and $(\rho_{0}^\ep, u_{0}^\ep)$ be given by \eqref{id-d-1} and \eqref{id-v-1} with 
	\begin{align}\label{R-lb}
		\min_{x\in \overline\Omega_\ep }   \rho_0^{\ep} = \underline{\rho}^\ep >0  .
	\end{align} 
	Then, for any $ T>0 $, there exists a unique global-in-time regular solution $(\rho^\ep,u^\ep) $ to the system \eqref{app_eq} with initial data $ (\rho_{0}^\ep, u_{0}^\ep)$ and boundary condition \eqref{app:bc} such that 
	\begin{align}
		(\rho^\ep ,u^\ep)\in  \mathcal{X}_2(T)\times \mathcal{Y}_2(T)\quad\text{with}\quad	\rho^\ep(t,x) \ge c\underline{\rho}^\ep >0\quad\text{for}\;\;(t,x)\in [0,T]\times \overline\Omega_\ep,\nonumber
	\end{align}
 where $c>0$ is an $\ep$-independent constant.
	Moreover, there exists $\lambda>0$ such that the solution satisfies the estimate  
    \equ{\label{est:BD_app}
		\Dt\calJ_{\tau,\lambda,\ep}(\rho^\ep,u^\ep) \le -\calD_{\tau,\lambda,\ep}(\rho^\ep,u^\ep) + a_\ep\left(\intM{ \rho |\pa_x \delta \calF_\ep(\rho)|^2}+\intM{\rho^\ep|u^\ep|^2}\right),
	}
    with $a_\ep\to 0$ as $\ep\to 0$.

    In the case when either $\tau>0$ or $\phi$ is uniformly bounded from below, we find $\lambda>0$ and $\ep_0>0$ such that for every $\ep\in(0,\ep_0)$:
    \equ{\label{est:BD_app_uniform}
		\Dt\calJ_{\tau,\lambda,\ep}(\rho^\ep,u^\ep) \le - (1-\tilde a_\ep)\calD_{\tau,\lambda,\ep}(\rho^\ep,u^\ep),
	}
    with $\calD_{\tau,\lambda,\ep}\ge 0$ for every $\ep\in(0,\ep_0)$ and $\tilde a_\ep \to 0$ as $\ep\to 0$.
    
\end{theorem}
The rest of this section is devoted to the proof of this theorem. Although local existence theory for regular solutions to system \eqref{app_eq} without the non-local terms is by now classical, it does not directly apply to our case. For this reason, we outline the proof of the following result in Appendix \ref{AppendixA}. 

\begin{proposition}\label{app-lt}
Let $(\rho_{0}^{\ep}, u_{0}^{\ep})$ be given by \eqref{id-d-1}, \eqref{id-v-1} and satisfy \eqref{R-lb}.  Moreover, set
\begin{align}
	A_\ep\coloneq \Vert \rho_0^{\ep} \Vert_{H^3(\Omega_\ep)}+ \Vert u_0^{\ep} \Vert_{H^3(\Omega_\ep)} <+\infty.\nonumber
\end{align}
Then, there exist some $T_0=T_0(\ep,A_\ep,\underline{\rho}^\ep)>0$ and a unique regular solution $(\rho^{\ep}, u^{\ep})$ to \eqref{app_eq} on the interval $[0,T_0)$.
\end{proposition} 

Let us now denote the maximal time of existence of the regular solution by $T_{\max}$. Our next step is to derive the estimates that are independent of $T_{\max}$ and in consequence allow us to extend the solution to arbitrary large time interval $[0,T)$. The crucial estimate here is the positive bound from below for the density, which is a consequence of the $\ep$-dependent approximation.

\subsection{Total energy and modulated energy esimates}\label{sec:energy-approx}

We begin by estimating the total energy $\calE_\ep$ along solutions of the approximate system. Due to the regularity of $(\rho^\ep,u^\ep)$, we can rigorously compute the temporal derivative of the total energy to obtain 
\begin{align*}
    \Dt \calE_\ep(\rho^\ep,u^\ep) = -\intM{\mu_\ep(\rho)|\pa_x u|^2}-\tau\intM{\rho u^2} - \frac12 \calK_\phie(\rho,u) \le 0.
\end{align*}
In particular, the energy is nonincreasing with
\equ{\label{energy}
    \calE_\ep(\rho^\ep(t),u^\ep(t)) \le \calE_\ep(\rho_0^\ep,u_0^\ep),
}
which provides uniform estimates for the kinetic energy and internal energy since the interaction potential $W$ is assumed to be bounded from below. Hence,
\[
    \calF(\rho^\ep(t)) = \calF_\ep(\rho^\ep(t)) + \frac{\ep^\beta}{1-\alpha}\intM{(\rho^\ep)^\alpha(t)} \le \sup_{\ep>0}\calE_\ep(\rho_0^\ep,u_0^\ep) + \frac{\ep^\beta}{1-\alpha}(2R_\ep)^{1-\alpha}.
\]
Choosing $\beta$ and $R_\ep$ appropriately, one obtains $\ep^\beta R_\ep^{1-\alpha}\to 0$, which consequently implies
\[
   c_\rho\coloneq \sup_{\ep>0}\sup_{t\in[0,T]} \|\rho^\ep(t)\|_{L^\gamma} <+\infty.
\]

To obtain the functional inequality \eqref{est:BD_app}, we perform similar computations as in the proof of Lemma~\ref{lem:bd-estimate} to arrive at the differential inequality
\begin{align*}
    \Dt\calJ_{\tau,\lambda,\ep} &\le -\calD_{\tau,\lambda,\ep} + \calR_\ep,
\end{align*}
with the error term $\calR_\ep\ge 0$, that can be estimated from above by
\[
    \calR_\ep \le 3c_\rho\Vert \phie-\phi \Vert_{L^{\gamma^\prime}} \left(  \intM{ \vr^\ep(t) \vert u^\ep(t) \vert^2 } + \intM{\rho^\ep(t) |\pa_x \delta \calF_\e(\rho^\ep(t))|^2}\right).
\]
Since $\phie$ converges to $\phi$ in $L^\gamma$, \eqref{est:BD_app} holds with $a_\ep\coloneq 3c_\rho\Vert \phie-\phi \Vert_{L^{\gamma^\prime}}$.

\medskip

Summarizing the above estimates, we obtain the following a-priori estimates 
\equ{\label{est_mod_app}
	&\sup_{t\in[0,T_{\max})}\lr{\intM{\vr^\ep(t)|u^\ep(t)|^2}+\calF(\rho^\ep(t))+ \intM{\vr^\ep(t)|\px\delta\calF_\ep(\rho^\ep(t))|^2}}\\
	&\qquad+\intTM{\mu_\ep(\vr^\ep)|\px u^\ep|^2} +\tau \intTM{\rho^\ep|u^\ep|^2}+\int_0^T \calK_\phi(\rho^\ep,u^\ep)\,\dt \le C,
}
for some $\ep$-independent constant $C>0$.

\subsection{Uniform estimates for the density}
To deduce a-priori estimates for the density from the lemmas above, we need to analyze the structure of the term $\delta\calF_\ep$. To do so, we first estimate the total mass and the higher-order moments of the density.

\begin{lemma}
	For regular solutions $(\rho,u)$ on time interval $[0,T_{\rm max})$, and for $\kappa>0$, we have
	\begin{subequations}\label{moments}
		\begin{gather}
    \intM{\vr(t,x)}=1\qquad\text{for all $t\in[0,T_{\max})$},\label{moments_0}\\
   \sup_{t\in[0,T_{\rm max})}\intM{|x|^2\vr(t,x)} \leq C(T_{\rm max},data),\label{moments_2}\\
   \Dt\intM{|x|^{\kappa+2}\vr(t,x)}\leq C_1(\kappa) \intO{|x|^{\kappa+2}\vr(t,x)}+C_2(\kappa)\intO{\vr(t,x)|u(t,x)|^{2+\kappa}}\label{moments_2k}.
		\end{gather}
	\end{subequations}

\end{lemma}
\begin{proof}
	Integrating the continuity equation \eqref{app_eq1} w.r.t.\ the spatial variable, and using the zero Dirichlet boundary condition, we obtain
	\[
		\frac{ \rm d}{\dt}\intM{ \vr(t,x)}=0,
    \]
	and so, integrating w.r.t.\ time, we see that  \eqref{moments_0} holds.
 
 Next, multiplying continuity equation \eqref{app_eq1} by $x^2$ and integrating by parts, we obtain
	\equ{\label{x2rho}
		\Dt\intM{x^2\vr}=2\intM{x \vr u}\leq \intM{x^2\vr}+\intO{\vr u^2},
	}
	and we conclude the proof of \eqref{moments_2} by the Gronwall argument using estimates from \eqref{est_mod_app} .
	
	At last, we multiply the continuity equation by $|x|^{2+\kappa}$ to check that 
	\[
		\Dt\intM{x^2|x|^\kappa\vr}=(2+\kappa)\intM{x|x|^\kappa \vr u}.
    \]
	Applying the Young inequality to the r.h.s., we directly obtain
	\equ{\label{xkrho2}
		\Dt\intM{|x|^{\kappa+2}\vr}
		&\leq C_1(\kappa) \intM{|x|^{\kappa+2}\vr}+C_2(\kappa)\intM{\vr|u|^{2+\kappa}}.\qedhere
	}
\end{proof}
\begin{corollary}\label{cor_bddr}
	For regular solutions, we have that
	\equ{\label{rho_updown}
		\sup_{t\in[0,T_{\max})}\Biggl\{\frac{\gamma^2}{(\gamma-\sfrac{1}{2})^2}\intM{\left|\px\rho^{\gamma-\sfrac{1}{2}}(t,x)\right|^2} + \frac{\ep^{2\beta}\alpha^2}{(\alpha-\sfrac{1}{2})^2}\intM{\left|\px \rho^{\alpha-\sfrac12}(t,x)\right|^2} \Biggr\}\leq C, 
	}
	where the constant $C>0$ is independent of $\ep>0$.
	
	In particular, there exists a constant $c_\infty>0$, independent of $\ep>0$ such that
	\equ{\label{norm_sup}
		\|\vr\|_{C([0,T_{\max}]\times \overline\Omega_\ep)}\leq c_\infty.
	}
	Moreover, for some $c_\ep>0$ depending on $\ep>0$, we have
	\equ{\label{norm_inf}
		\|\vr^{-1}\|_{C([0,T_{\max}]\times\overline\Omega_\ep)}\leq c_\ep,\qquad\text{with\; $c_\ep\to+\infty$ as $\ep\to 0$.}
	}
\end{corollary}
\begin{proof}
	From Lemma~\ref{moments} and the fact that  
 \[
 |\pa_x W * \rho(x)| \leq C\lt( 1 + |x| + \int_\R |x|\,\rho(x)\,\dx\rt) \leq C(1+|x|),
 \]
 we conclude from \eqref{moments_2} that
	\equ{\label{est_potw}
		\intM{\rho(x)|\px W*\rho(x)|^2}\leq C\lr{1 +\intM{|x|^2\vr}}\leq C.
	}
	Therefore, we may now deduce from \eqref{est_mod_app} that
	\equ{\nonumber
		\sup_{t\in[0,T_{\max})}\intM{\rho|\px\varphi_\ep(\rho)|^2}=
		\sup_{t\in[0,T_{\max})}\intM{\biggl|\frac{\gamma}{\gamma-\sfrac12}\px\rho^{\gamma-\frac{1}{2}}+\frac{\ep^\beta\alpha}{\alpha-\sfrac12}\px \rho^{\alpha-\frac12}\biggr|^2}\leq C.
	}
	Note that since 
	\[
	\frac{\gamma}{\gamma-\sfrac12}\frac{\ep^\beta\alpha}{\alpha-\sfrac12}\px\rho^{\gamma-\sfrac{1}{2}}\px \rho^{\alpha-\sfrac12} = \ep^\beta\gamma\alpha \rho^{\gamma+\alpha-3} |\partial_x\rho|^2 \ge 0,
	\]
	this estimate directly implies \eqref{rho_updown}.
	
	Using this estimate, we first see how the density is bounded from below. From \eqref{moments_0}, we have
	\[
	\intM{\rho(t,x)} = 1\quad\text{for all $t\in[0,T_{\max})$}.
	\]
	Since $\rho$ is a continuous function, there exist some $x_\ep(t)\in \Omega_\ep$ for each $t\in[0,T_{\max})$ such that
	\equh{\rho(t,x_\ep(t)) = \frac{1}{2R_\ep}.}
	On the other hand, for every $(t,x)\in[0,T_{\max})\times \Omega_\ep$, we can estimate
	\begin{align*}
	\big( \rho(t,x) \big)^{\alpha-1/2} - \big( \rho(t,\bar x_\ep(t)) \big)^{\alpha-1/2}
	&= \int_{x_\ep(t)}^x  \partial_y \rho^{\alpha-1/2}(t,y)\,\dy \\
 &\le \sqrt{|x-\bar x_\ep(t)|}\|\partial_x \rho^{\alpha-\sfrac{1}{2}}(t,\cdot) \|_{L^2(\Omega_\ep)},
	\end{align*}
	which then implies
	\begin{align*}
		\big( \rho(t,x) \big)^{\alpha-1/2}
		& \leq \left(\frac{1}{2R_\ep}\right)^{\alpha-1/2} + \sqrt{2R_\ep}\, \|\partial_x \rho^{\alpha-\sfrac{1}{2}}(t,\cdot) \|_{L^2(\Omega_\ep)}.
	\end{align*}
 Finally, owing to the second part of \eqref{rho_updown}, we obtain
	\[
	\rho^{-1}(t,x) \leq c_\ep  \qquad\text{for all $(t,x)\in[0,T_{\max})\times\overline\Omega_\ep$},
	\]
 for some constant $c_\ep$ with $c_\ep\to+\infty$ as $\ep\to 0$, thus proving \eqref{norm_inf}. 
	
Proving the estimate \eqref{norm_sup} is very similar, and uses the first part of the estimate \eqref{rho_updown}. Indeed, as before, we have that
	\equ{\nonumber
		\rho^\gamma(t,x) &= \rho^\gamma(t,x_\ep(t)) +\int_{x_\ep(t)}^x\pa_y \rho^{\gamma}(t,y)\, \dy \\
		&= \rho^\gamma(t,x_\ep(t)) + \frac{\gamma}{\gamma-\sfrac12}\int_{x_\ep(t)}^x\sqrt{\rho(t,y)}\,\pa_y \rho^{\gamma-\sfrac12}(t,y) \d y \\
		&\le \left(\frac{1}{2R_\ep}\right)^\gamma + \frac{\gamma}{\gamma-\sfrac12}\|\partial_x\rho^{\gamma-\sfrac{1}{2}}(t,\cdot)\|_{L^2(\Omega_\ep)}.
	}
	The last term is again controlled thanks to \eqref{rho_updown}, and so the proof of \eqref{norm_sup} is complete. Observe that the previous estimate is uniform w.r.t.\ $\ep>0$.
\end{proof}

\subsection{Conclusion of the proof of Theorem~\ref{Th:3}}
With these estimates at hand, we may extend our local-in-time solution to a global in-time regular solution, and the proof for the first part of Theorem \ref{Th:3} is complete, and the estimates above are valid for any $T>0$ instead of $T_{\rm max}$.

    As for the second part when either $\tau>0$ or $\phi$ is uniformly bounded from below, then $\lambda>0$ can be chosen, independently of $\ep>0$, such that $\ell_{\tau,\lambda}>0$. We then choose $\ep_0\ll 1$ small such that
    \[
        a_\ep < \min\left\{\frac{1}{2}, \ell_{\tau,\lambda}\right\}\qquad\text{for every $\ep\in(0,\ep_0)$},
    \]
    which can be done since $a_\ep\to 0$ as $\ep\to 0$. Since the term on the right-hand side of \eqref{est:BD_app} is contained in $\calD_{\tau,\lambda,\ep}$, we can then absorb them to obtain \eqref{est:BD_app_uniform}, therewith concluding the proof.

\section{Compactness results for approximate solutions}\label{sec:compactness}

The estimates obtained in Section~\ref{Sec:existence} (cf.\ \eqref{est_mod_app}) are insufficient to pass to the limit in the convective term in \eqref{app_eq}, which, at this point, is only known to be bounded in $L^1$ due to kinetic energy estimate \eqref{energy}. This section aims at improving the a-priori information and to deduce the convergence of the approximating sequence.
\subsection{Derivation of the Mellet-Vasseur inequality}
We start with the better velocity estimate in the spirit of the Mellet-Vasseur estimate \cite{MV07}, but on a  bounded domain $\Omega_\varepsilon=(-R_\varepsilon,R_\varepsilon)$. Note that due to zero Dirichlet boundary condition for $u^\ep$, this is indeed possible, and we have the following.
\begin{lemma}\label{lem:MV}
For any $\kappa\in (0,2\alpha]$, let  
\[
\calE_{\kappa}(\rho,u) \coloneq \frac{1}{2+\kappa}\intM{\rho \vert u \vert^{2+\kappa}}.
\]
Then, for regular approximate solutions, we have 
\equ{\nonumber
	& \sup_{t\in[0,T]}\calE_{\kappa}(t) + (1+\kappa) \intTM{\mu_\ep(\rho) |u|^\kappa \vert \pa_x u \vert^2} + \tau\intM{\rho |u|^{k+2}}
	\leq c(\vr_0,u_0,T) .
}
\end{lemma}
\begin{proof}
Multiplying the momentum equation \eqref{app_eq2} by $u |u|^\kappa$, multiplying the continuity equation \eqref{app_eq1} by $|u|^{2+\kappa}/(2+\kappa)$, and adding up the obtained expressions integrated over space, we obtain
\equ{\label{testMV}
	&\Dt\intM{\frac{\vr|u|^{2+\kappa}}{2+\kappa}}+(\kappa+1)\intM{\mu_\ep(\vr)|u|^\kappa|\px u|^2 } + \tau\intM{\rho |u|^{k+2}}\\
	&\hspace{2cm}=-\intM{\rho\px\varphi_\ep(\vr) u |u|^\kappa}-\intM{\rho\px W*^{\ep}\rho u |u|^\kappa}+\intM{\rho u|u|^\kappa[\phi_\ep*^{\ep},u]\rho }\\
	&\hspace{2cm}=I_1+I_2+I_3 .
}
Thanks to zero Dirichlet boundary condition for $u$, $I_1$ can be integrated by parts and estimated as
\equ{I_1&=-\intM{\px(\vr^\gamma+\ep\vr^\alpha) u|u|^\kappa}
	=(\kappa+1)\intM{(\vr^\gamma+\ep\rho^\alpha)|u|^\kappa\px u }\\
	&\leq
	\frac{(\kappa+1)}{2}\intM{(\gamma\vr^\gamma+\ep\alpha\vr^\alpha)|u|^\kappa|\px u|^2}+ c(\kappa,\alpha,\gamma)\intM{(\vr^\gamma+\ep\rho^\alpha)|u|^\kappa}.
}
The first term can be absorbed into the left-hand side of \eqref{testMV}, while to control the second term, we use Young's inequality to write for $\kappa\leq 2\alpha$,
\eqh{ \nonumber
\intM{\vr^\gamma|u|^\kappa}&= \intM{\lr{\vr^{\gamma-\frac{\kappa}{2}}+\ep\vr^{\alpha-\frac{\kappa}{2}}}\vr^{\frac{\kappa}{2}}|u|^\kappa}\\
	&\leq \intM{\vr u^2}+ c\intM{\vr^{\frac{2\gamma-\kappa}{2-\kappa}}}+c\intM{\ep^{\frac{2}{2-\kappa}}\vr^{\frac{2\alpha-\kappa}{2-\kappa}}}
	.}
In the inequality above, the first term can be controlled by the energy estimate \eqref{est_mod_app}, the second one is bounded by the $L^1\cap L^\infty$ norm of $\vr$, provided that $\kappa\leq 2\alpha$, which is bounded on the account of \eqref{moments_0} and \eqref{norm_sup} (independently of $\varepsilon$), and the third one disappears when $\ep\to0$ also on account of \eqref{norm_sup}.

For term $I_3 $, we observe that
\begin{align*}
	I_3&=  \iintM{ u(x) |u(x)|^\kappa \phi_\ep (x-y) (u(y)-u(x)) \rho(x) \rho(y)}\\
	& =-  \iintM{u(y) |u(y )|^\kappa \phi_\ep(x-y) (u(y)-u(x)) \rho(x) \rho(y) }, 
\end{align*}
and so we have
\begin{align*}
	I_3&= \frac12 \iintM{\phi_\ep(x-y)\lr{ u(x) |u(x)|^\kappa-u(y) |u(y)|^\kappa}  (u(y)-u(x)) \rho(x) \rho(y)}\leq 0,
\end{align*}
due to the monotonicity of $r\mapsto r |r|^\kappa$.

The last difficulty is to deal with the term $I_2$, and in particular with the attractive part of the potential $W$, for which $|\pa_x W * \rho | \leq c(1+ |x|)$. The constant $c$ depends on the initial data and $T$ via the estimate of $\intM{|x|\rho}$ which follows from \eqref{moments_0} and \eqref{moments_2}. Therefore, we have
\eqh{
	I_2 \leq c\intM{(1+|x|)\vr |u|^{\kappa+1}}&=c\intM{\lr{(1+|x|)^{2+\kappa}\vr}^{\frac{1}{2+\kappa}} \lr{\vr^{\frac{1}{2+\kappa}}|u|}^{\kappa+1}}\\
	&\leq c\intM{(1+|x|)^{2+\kappa}\vr}+\intM{\vr|u|^{\kappa+2}},
} 
where we used Young's inequality  with $p=\kappa+2$, $p'=\frac{\kappa+2}{\kappa+1}$.
The second term can be estimated by the l.h.s. of \eqref{testMV} and the Gronwall argument. To estimate the first one, we use the higher order moment estimate for $\rho$, \eqref{moments_2k}, and then we combine it with the Gronwall argument.
\end{proof}

The rest of this section is devoted to proving the following result.

\begin{theorem} Let $\{(\rho^\ep,\sqrt{\rho^\ep}u^\ep)\}_{\ep>0}$ be a sequence of regular solutions to \eqref{app_eq}. Then there exists a limit pair $(\rho,\sqrt{\rho}u)\in L^\infty((0,T)\times\R)\times L^{2+\kappa}((0,T)\times\R)$ and a (not relablled) subsequence such that
\begin{align}\label{eq:convergences-lsc}
	\left\{\qquad \begin{gathered}
		(\tilde\rho^\ep)^\gamma \to \rho^\gamma,\quad \tilde\rho^\ep u^\ep \to  m\quad\text{strongly in $L^1((0,T)\times\R)$,}\\
		\frac{\tilde\rho^\ep u^\ep}{\sqrt{\tilde\rho^\ep}} \to \frac {m}{\sqrt{\rho}}=\sqrt{\rho}u\quad\text{strongly in $L^2(0,T;L^2(\R))$},\\
		\tilde\rho^\ep(t)\,W^-\ast\tilde\rho^\ep(t) \to \rho(t)\, W^-\ast\rho(t)\quad\text{strongly in $L^1(\R)$ for almost every $t\ge 0$},\\
		\tilde\rho^\ep\,\partial_x W\ast^\ep\tilde\rho^\ep \rightharpoonup \rho\,\partial_x W\ast\rho\quad\text{weakly in $L^1(0,T;L_{loc}^1(\R))$},\\
		\phi_\ep\ast^\ep \tilde\rho^\ep \rightharpoonup^* \phi\ast \rho\quad\text{weakly-$*$ in $L^\infty(0,T;L_{loc}^\infty(\R))$}.
	\end{gathered}\right.
\end{align}
Here, $\tilde\rho^\ep = \mathbbm{1}_{\Omega_\ep}\rho^\ep$ is the extension of $\rho^\ep$ by zero outside of $\Omega_\ep$ and $\ep$ and $R_\ep$ satisfies the following hypthesis: 
\begin{align}\label{ep-R-1}
    \sup_{\ep>0}\,\ep^\beta R_\ep^{1-\alpha}<+\infty.
\end{align}
\end{theorem}

\subsection{Convergence of the density}

\subsubsection{Strong convergence of $(\rho^\ep)^\gamma$}
From the energy estimate \eqref{energy}, we have that
\[
\sup_{\ep>0}\sup_{t\in[0,T]} \int_{\R} \int_0^{\tilde\rho^\ep(t,x)}\varphi(r)\d r\,\dx \le \frac{\ep^\beta}{1-\alpha}(2R_\ep)^{1-\alpha} + c,
\]
where $c>0$ is a constant uniform in $\ep>0$. In particular, if $\sup_{\ep>0}\,\ep^\beta R_\ep^{1-\alpha}<+\infty$, then 
\[
\sup_{\ep>0}\sup_{t\in[0,T]} \|(\tilde\rho_t^\ep)^\gamma\|_{L^1} <+\infty.
\]
Furthermore, from the bound on the derivative of the pressure, we obtain for all $\psi\in C^1(\R)$,
\begin{align*}
\int_\R \partial_x \psi \,(\tilde\rho^\ep(t))^\gamma\,\dx &= \int_{\Omega_\ep} \partial_x \psi\,(\rho^\ep(t))^\gamma\,\dx \\
&\hspace{-4em}= - \int_{\Omega_\ep} \psi\, \partial_x (\rho^\ep(t))^\gamma\,\dx + \psi(R_\ep)(\rho^\ep(t,R_\ep))^\gamma - \psi(-R_\ep)(\rho^\ep(t,-R_\ep))^\gamma \\
&\hspace{-4em}= - \frac{\gamma}{\gamma-\sfrac{1}{2}}\int_{\Omega_\ep} \psi\, \sqrt{\rho^\ep(t)}\,\partial_x (\rho_t^\ep)^{\gamma-\sfrac{1}{2}}\,\dx + \psi(R_\ep)(\rho^\ep(t,R_\ep))^\gamma - \psi(-R_\ep)(\rho^\ep(t,-R_\ep))^\gamma \\
&\hspace{-4em}\le c\|\psi\|_{\sup}\left( \|\rho_0\|_{L^1} + \sup_{\ep>0}\sup_{t\in[0,T]}\|\partial_x (\rho^\ep(t))^{\gamma-\sfrac{1}{2}}\|_{L^2(\Omega_\ep)}^2\right) <+\infty.
\end{align*}
Taking the supremum over all $\psi\in C^1(\R)$ with $\|\psi\|_{\sup}\le 1$ yields
\[
\sup_{\ep>0}\sup_{t\in[0,T]}\|D(\tilde\rho^\ep)^\gamma\|_{TV} <+\infty,
\]
which together with the previous bound implies that
\[
\bigl((\tilde\rho^\ep)^\gamma \bigr)_{\ep>0} \subset L^\infty(0,T;BV(\R))\quad\text{is bounded}.
\]
We now prove the equi-continuity property of the family of curves $\{t\mapsto (\tilde\rho_t^\ep)^\gamma\}_{\ep>0}$. Owing to the regularity of the curve $t\mapsto \rho_t^\ep$, we can take the temporal derivative to deduce
\begin{align*}
\frac{\d}{\dt} \int_\R \psi\,(\tilde\rho^\ep)^\gamma\,\dx &= \gamma \int_{\Omega_\ep} \psi\, (\rho^\ep)^{\gamma-1}\,\partial_x(\rho^\ep u^\ep)\,\dx \\
&= \frac{\gamma}{\gamma-\sfrac{1}{2}}\int_{\Omega_\ep} \psi\, \sqrt{\rho^\ep} u_t^\ep\,\partial_x(\rho^\ep)^{\gamma-\sfrac{1}{2}}\,\dx + \int_{\Omega_\ep} \psi\, \mu(\rho^\ep)\, \partial_x u^\ep\,\dx \\
&\le c_\gamma \|\psi\|_{L^\infty}\Bigl( \|u^\ep\|_{L^2(\tilde\rho^\ep)}\|\partial_x(\rho^\ep)^{\gamma-\sfrac{1}{2}}\|_{L^2(\Omega_\ep)} + \|\tilde\rho^\ep\|_{L^\gamma}^\gamma \|\partial_x u^\ep\|_{L^2(\mu(\tilde\rho_t^\ep))}\Bigr),
\end{align*}
which, due to the energy estimate \eqref{energy} and \eqref{rho_updown}, implies that the sequence
\[
\bigl(\partial_t (\tilde\rho_t^\ep)^\gamma\bigr)_{\ep>0} \subset L^2(0,T;L^1(\R))  \quad\text{is bounded}.
\]
Since the embedding $BV(\R)\hookrightarrow L^1(\R)$, the Aubin--Lions lemma then yields the relative compactness of the sequence $((\tilde\rho^\ep)^\gamma)_{\ep>0}$ in $C([0,T];L^1(\R))$, i.e.\ there exists a curve $\chi\in C([0,T];L^1(\Omega))\cap L^\infty(0,T;BV(\R))$ and a (not relabelled) subsequence, such that
\[
(\tilde\rho^\ep(t))^\gamma \to \chi(t)\quad\text{strongly in $L^1(\R)$ and pointwise-a.e.\ for every $t\in[0,T]$}.
\] 

\subsubsection{Strong convergence result} Setting $\rho(t)\coloneq\chi^{1/\gamma}(t)$ for every $t\in[0,T]$, we notice that also $\tilde\rho^\ep(t)\to \rho(t)$ pointwise. Since $(\tilde\rho^\ep(t))_{\ep>0}$ is tight and the uniform $L^\gamma$-integrability gives
\[
\int_B \rho^\ep\,\dx \le 
|B|^{1-\sfrac{1}{\gamma}}\|\rho^\ep\|_{L^\gamma}\qquad\text{for every Borel set $B\subset\R$},
\]
this implies that $(\tilde\rho^\ep(t))_{\ep}$ is equi-absolutely continuous for every $t\in[0,T]$, and we conclude with the Vitali convergence theorem that
\[
\tilde\rho^\ep(t) \to \rho(t)\quad\text{strongly in $L^1(\R)$ and pointwise-a.e.\ in $\R$ for every $t\in[0,T]$}.
\]
In particular, we obtain the strong convergence in every $L^p(\R)$, $p\ge 1$.

\subsection{Convergence of the momentum} 

\subsubsection{Weak convergence result} We begin with a weak convergence result for the sequence of approximate momentum $(m^\ep\coloneq\tilde\rho^\ep u^\ep)_{\ep>0}$, relying only on energy estimates obtained in \eqref{sec:energy-approx} and Lemma~\ref{lem:MV}. From the energy estimate \eqref{energy}, Lemma~\ref{lem:MV}, and \eqref{norm_sup} we easily find that
\[
(m^\ep)_{\ep>0}\subset L^\infty(0,T;L^1\cap L^{1+\kappa}(\R))\quad\text{is a bounded sequence.}
\]
Moreover, using the momentum equation, we have for all $\psi\in \text{Lip}_b(\R)$,
\begin{align*}
\frac{\d}{\dt}\intO{\psi\,m^\ep} 
&\le \|\partial_x\psi\|_{\sup}\Bigl(\|u^\ep\|_{L^2(\rho^\ep)}^2 +  \sqrt{\|\mu_\ep(\rho_t^\ep)\|_{L^1}}\|\partial_x u^\ep\|_{L^2(\mu_\ep(\rho^\ep))}\Bigr) \\
&\hspace{4em}+ \|\psi\|_{\sup}\Bigl( \sqrt{\|\phi_\ep *^\ep \rho^\ep\|_{L^\infty}\calK_\phi(\rho^\ep,u^\ep)} + \|\partial_x\delta\calF_\ep(\rho^\ep)\|_{L^2(\rho^\ep)} + \tau\|u^\ep\|_{L^2(\rho^\ep)}\Bigr),
\end{align*}
thereby implying that the sequence
\[
(\partial_t m^\ep)_{\ep>0} \subset L^2(0,T;\text{Lip}_b(\R)')\quad\text{is bounded}.
\]
Consequently, we obtain from an Arzela-Ascoli-type argument, the existence of a (not relabelled) subsequence and a curve $m\in C([0,T];L^1\cap L^{1+\kappa}(\R))$ such that
\[
m^\ep(t) \rightharpoonup m(t)\quad\text{weakly in $L^1\cap L^{1+\kappa}(\R)$ for every $t\in[0,T]$.}
\]
Owing to the weak-$*$ lower semicontinuity of the functional (cf.\ Proposition~\ref{prop:lsc} below)
\[
(\rho,m)\mapsto \begin{cases}\displaystyle
\intO{\frac{|m|^2}{\rho}} &\text{if\; $m\,\dx \ll \rho\,\dx$},\\
+\infty &\text{otherwise},
\end{cases}
\]
we obtain
\[
\intO{\frac{|m(t)|^2}{\rho(t)}} \le \liminf_{\ep\to 0} \intO{\frac{|m^\ep(t)|^2}{\tilde\rho^\ep(t)}} <+\infty\qquad\text{for every $t\in[0,T]$}.
\]
Therefore, we can define
\equ{\nonumber
u(t,x)\coloneq \begin{cases}\displaystyle
\frac{m}{\rho} &\text{for\; } \rho(t,x)>0,\\
0 &\text{for\; } \rho(t,x)=0.
\end{cases}\qquad \text{for every $t\in[0,T]$}.
}

\subsubsection{Strong convergence result for an auxiliary sequence}
To prove the strong convergence of the momentum, we first prove the convergence of the sequence $((\tilde\rho^\ep)^{\gamma-\sfrac{1}{2}}m^\ep)_{\ep>0}$.

For almost every $t\in[0,T]$, we easily deduce
\[
\sup_{\ep>0}\int_0^T \|(\tilde\rho^\ep)^{\gamma-\sfrac{1}{2}}m^\ep\|_{L^1}^2\,\dt 
\le c_\infty^{(2\gamma-1)} \sup_{\ep>0}\int_0^T \|u^\ep\|_{L^2(\rho^\ep)}^2\,\dt <+\infty.
\]
Furthermore, we obtain, for almost every $t\in[0,T]$,
\begin{align*}
\int_\R \partial_x\psi\, (\tilde\rho^\ep)^{\gamma-\sfrac{1}{2}}m^\ep\,\dx &= \int_{\Omega_\ep} \partial_x\psi\, (\rho^\ep)^{\gamma+\sfrac{1}{2}} u^\ep\,\dx \\
&= -\frac{\gamma+\sfrac{1}{2}}{\gamma}\int_{\Omega_\ep} \psi\,\sqrt{\rho^\ep} u^\ep\,\partial_x (\rho^\ep)^{\gamma}\,\dx - \frac{1}{\gamma}\intM{ \psi\,\sqrt{\rho^\ep}\,\mu(\rho^\ep)\,\partial_x u^\ep} \\
&\le c_\gamma\sqrt{c_\infty}\|\psi\|_{\sup}\Bigl( \|u^\ep\|_{L^2(\rho^\ep)}\|\partial_x (\rho_t^\ep)^{\gamma-\sfrac{1}{2}}\|_{L^2(\Omega_\ep)} + \|\tilde\rho^\ep\|_{L^\gamma}^\gamma\|\partial_x u_t^\ep\|_{L^2(\mu(\rho^\ep))}\Bigr),
\end{align*}
for some constant $c_\gamma>0$ depending only on $\gamma\ge 1$. Therefore, taking the supremum over all $\psi\in C^1(\R)$ with $\|\psi\|_{\sup}\le 1$, squaring and then integrating over time yields
\[
\sup_{\ep>0}\int_0^T \|(\tilde\rho^\ep)^{\gamma-\sfrac{1}{2}}m^\ep\|_{TV}^2\,\dt <+\infty,
\]
which, together with the previous estimate, implies that
\[
\bigl((\tilde\rho^\ep)^{\gamma-\sfrac{1}{2}}m^\ep\bigr)_{\ep>0} \subset L^2(0,T;BV(\R))\quad\text{is bounded}.
\]

As in the prove for the convergence of the density, we need to show the equi-continuity property of the curves $\{t\mapsto (\tilde\rho^\ep)^{\gamma-\sfrac{1}{2}}m^\ep\}_{\ep>0}$ in an appropriate function space. Recall that $u^\ep$ satisfies
\begin{align*}
\partial_t u^\ep + u^\ep\,\partial_x u^\ep = -\partial_x\delta\calF_\ep(\rho^\ep) + \frac{1}{\rho^\ep} \partial_x(\mu_\ep(\rho^\ep)\partial_x u^\ep) + [\phi_\ep\ast^\ep,u^\ep]\rho^\ep 
\end{align*}
for $\rho^\ep$-almost every $(t,x)\in(0,T)\times\Omega_\ep$. For any $\psi\in \text{Lip}_b(\R)$, we then compute to obtain
\begin{align*}
\frac{\d}{\dt}\int_\R \psi\, (\tilde\rho^\ep)^{\gamma-\sfrac{1}{2}}m^\ep\,\dx &= \int_{\Omega_\ep} \psi\, u^\ep\,\partial_t (\rho^\ep)^{\gamma+\sfrac{1}{2}}\,\dx + \int_{\Omega_\ep} \psi \,(\rho^\ep)^{\gamma+\sfrac{1}{2}}\,\partial_t u^\ep\,\dx \\
&\hspace{-8em}= -(\gamma+\sfrac{1}{2})\int_{\Omega_\ep} \psi\,u^\ep (\rho^\ep)^{\gamma-\sfrac{1}{2}}\partial_x m^\ep\,\dx - \int_{\Omega_\ep} \psi\, u^\ep (\rho^\ep)^{\gamma+\sfrac{1}{2}} \partial_x u^\ep\,\dx \\
&\hspace{-8em}-\int_{\Omega_\ep} \psi\,(\rho^\ep)^{\gamma+\sfrac{1}{2}}\partial_x\delta\calF_\ep(\rho^\ep)\,\dx + \int_{\Omega_\ep} \psi\,(\rho^\ep)^{\gamma-\sfrac{1}{2}} \partial_x(\mu_\ep(\rho^\ep)\partial_x u^\ep)\,\dx + \int_{\Omega_\ep} \psi\, (\rho^\ep)^{\gamma+\sfrac{1}{2}}[\phi\ast,u^\ep]\rho^\ep\,\dx \\
&\hspace{-8em}= -(\gamma-\sfrac{1}{2}) \int_{\Omega_\ep} \psi\, u^\ep (\rho^\ep)^{\gamma+\sfrac{1}{2}} \partial_x u^\ep\,\dx + \int_{\Omega_\ep} \partial_x\psi\,(u^\ep)^2\,(\rho^\ep)^{\gamma+\sfrac{1}{2}} \, \dx -\int_{\Omega_\ep} \psi\,(\rho^\ep)^{\gamma+\sfrac{1}{2}}\partial_x\delta\calF_\ep(\rho^\ep)\,\dx \\
&\hspace{-8em} -\int_{\Omega_\ep} \partial_x\psi\,(\rho^\ep)^{\gamma-\sfrac{1}{2}} \mu_\ep(\rho^\ep)\,\partial_x u^\ep\,\dx - \int_{\Omega_\ep} \psi\,\partial_x(\rho^\ep)^{\gamma-\sfrac{1}{2}} \mu_\ep(\rho^\ep)\,\partial_x u^\ep\,\dx + \int_{\Omega_\ep} \psi\, (\rho^\ep)^{\gamma+\sfrac{1}{2}}[\phi\ast,u^\ep]\rho^\ep\,\dx, \\
&\hspace{-8em}= -(\gamma-\sfrac{1}{2})I_1 + I_2 - I_3 - I_4 - I_5 + I_6.
\end{align*}
where each of the terms above can be estimated by
\begin{align*}
I_1 &=\intM{\psi\, \sqrt{\rho^\ep}u^\ep\, (\rho^\ep)^{\gamma} \partial_x u^\ep} \le c_\infty^{\gamma}\|\psi\|_{\sup}\|u^\ep\|_{L^2(\rho^\ep)}\|\partial_x u^\ep\|_{L^2(\mu(\rho^\ep))} \\
I_2 &=\int_{\Omega_\ep} \partial_x\psi\,(u^\ep)^2\,(\rho^\ep)^{\gamma+\sfrac{1}{2}} \, \dx \le c_\infty^\gamma\|\partial_x\psi\|_{\sup}\| u^\ep\|_{L^2(\rho^\ep)}\\
I_3 &=\int_{\Omega_\ep} \psi\,(\rho^\ep)^{\gamma}\sqrt{\rho^\ep}\partial_x\delta\calF_\ep(\rho^\ep)\,\dx \le c_\infty^{\gamma}\|\psi\|_{\sup}\|\partial_x\delta\calF_\ep(\rho^\ep)\|_{L^2(\rho^\ep)} \\
I_4 &=\int_{\Omega_\ep} \partial_x\psi\,(\rho^\ep)^{\gamma-\sfrac{1}{2}} \mu_\ep(\rho^\ep)\,\partial_x u^\ep\,\dx \le c_\infty^{\gamma-\sfrac{1}{2}}\|\partial_x\psi\|_{\sup}\sqrt{\|\mu_\ep(\tilde\rho^\ep)\|_{L^1}}\|\partial_x u^\ep\|_{L^2(\mu(\rho_t^\ep))} \\
I_5 &=\int_{\Omega_\ep} \psi\,\partial_x(\rho^\ep)^{\gamma-\sfrac{1}{2}} \mu_\ep(\rho^\ep)\,\partial_x u^\ep\,\dx \le \gamma c_\infty^\gamma\|\psi\|_{\sup} \|\partial_x(\rho^\ep)^{\gamma-\sfrac{1}{2}}\|_{L^2(\Omega_\ep)}\|\partial_x u^\ep\|_{L^2(\mu_\ep(\rho_t^\ep))} \\
I_6 &= \int_{\Omega_\ep} \psi\, (\rho^\ep)^{\gamma+\sfrac{1}{2}}[\phi_\ep\ast^\ep,u^\ep]\rho^\ep\,\dx \le c_\infty^{\gamma-\sfrac{1}{2}}\|\psi\|_{\sup}\sqrt{\|\phi_\ep *^\ep \rho^\ep\|_{L^\infty(\Omega_\ep)}\calK_\phi(\rho^\ep,u_t^\ep)}.
\end{align*}
Since all the terms on the right-hand side are 2-integrable in time, we deduce that
\[
\bigl(\partial_t[(\tilde\rho^\ep)^{\gamma-\sfrac{1}{2}}m^\ep]\bigr)_{\ep>0} \subset L^2(0,T; \text{Lip}_b(\R)').
\]
With the compact embedding $BV(\R)\hookrightarrow L^1(\R)$ and the continuity of $L^1(\R)\hookrightarrow \text{Lip}_b(\R)'$, we again apply the Aubin-Lions lemma to obtain the existence of a (not relabelled) subsequence and a curve $\eta\in L^2([0,T];L^1(\R))\cap L^2([0,T];BV(\R))$, such that
\[
(\tilde\rho^\ep)^{\gamma-\sfrac{1}{2}}m^\ep \to \eta\quad\text{strongly in $L^2(0,T;L^1(\R))$ and pointwise-a.e.\ in $[0,T]\times\R$}.
\] 
To identify the limit $\eta$, we simply notice that for every $\psi\in C_c([0,T]\times \R)$, we have that
\begin{align*}
\intTO{\psi\,\eta} = \lim_{\ep\to 0} \intTO{\psi\, (\tilde\rho^\ep)^{\gamma-\sfrac{1}{2}}m^\ep} = \intTO{\psi\, \rho^{\gamma-\sfrac{1}{2}}\, m },
\end{align*}
where the second equality follows from the strong convergence of $\rho^\ep$ in any $L^p$-space and the weak convergence of $m^\ep$ in $L^{2+\kappa}([0,T]\times\R)$. By density, we then conclude that
\[
(\tilde\rho^\ep)^{\gamma-\sfrac{1}{2}}m^\ep \to \rho^{\gamma-\sfrac{1}{2}}\, m \quad\text{strongly in $L^2([0,T];L^1(\R))$ }
\]  
and 
\[(\tilde\rho^\ep)^{\gamma-\sfrac{1}{2}}m^\ep \to \rho^{\gamma+\sfrac{1}{2}}\, u \quad \text{a.e.\ in $[0,T]\times\R$}.\]

\subsubsection{Strong convergence of $m^\ep/\sqrt{\tilde\rho^\ep}$} 

An important consequence of the strong convergence of the sequence $(\tilde\rho^\ep)^{\gamma-\sfrac{1}{2}}m^\ep$ and of the uniform (in $\ep>0$) Mellet-Vasseur estimate from Lemma~\ref{lem:MV} is the convergence 
\begin{align}\label{eq:kin-dens-conv}
g^\ep\coloneq\sqrt{\tilde\rho^\ep}u^\ep \to  \sqrt{\rho}u =: g \quad\text{strongly in $L^2([0,T]\times\R)$},
\end{align}
which we now prove, resulting in the convergence of the convective term.

Consider, as before, the set $A_\delta\coloneq\{(t,x)\in [0,T]\times\R:\rho_t(x)>\delta\}$, $\delta>0$. Then,
\[
g^\ep(t,x)=\frac{(\tilde\rho^\ep(t,x))^{\gamma-\sfrac{1}{2}}m^\ep(t,x)}{(\tilde\rho^\ep(t,x))^{\gamma}}\to \frac{(\rho(t,x))^{\gamma+\sfrac{1}{2}}\, u(t,x)}{(\rho(t,x))^{\gamma}} = g(t,x)\quad\text{for a.e.\ $(t,x)\in A_\delta$.}
\]
We then use Lemma~\ref{lem:MV} and the second moment estimate to obtain
\begin{align*}
\iint_{[0,T]\times B_n^c} |g^\ep|^2 \mathbbm{1}_{A_\delta}\,\dx\,\dt 
&\le n^{-\frac{2\kappa}{2+\kappa}}\left( \iint_{[0,T]\times \R} |x|^2\,\tilde\rho^\ep\dx\,\dt\right)^{\frac{\kappa}{2+\kappa}}\|u^\ep\|_{L^{2+\kappa}([0,T]\times\R,\tilde\rho^\ep)}^2,
\end{align*}
which implies the existence of some constant $c_2>0$, independent of $n\ge 1$, such that
\[
\sup_{\ep>0} \iint_{[0,T]\times B_n^c} |g^\ep|^2 \mathbbm{1}_{A_\delta}\,\dx\,\dt \le c_2 n^{-\frac{2\kappa}{2+\kappa}}\qquad\forall\,n\ge 1,
\]
i.e., $(g^\ep)_{\ep>0}$ is uniformly integrable in $L^2(A_\delta)$. Using again Lemma~\ref{lem:MV} together with the uniform $L^\gamma$-bound of $(\tilde\rho^\ep)_{\ep>0}$, we obtain for every Borel set $B\subset[0,T]\times\R$,
\begin{align}\label{smalness}
\iint_B |g^\ep|^2 \mathbbm{1}_{A_\delta}\,\dx\,\dt \le |B|^{\frac{\gamma}{\gamma-1}\frac{\kappa}{2+\kappa}}\|\tilde\rho^\ep\|_{L^\gamma([0,T]\times\R)}^{\frac{\kappa}{2+\kappa}}\|u^\ep\|_{L^{2+\kappa}([0,T]\times\R,\tilde\rho^\ep)}^2,
\end{align}
which gives the equi-absolute continuity of $(g^\ep)_{\ep>0}$ in $L^2(A_\delta)$, and therewith the convergence
\begin{align*}\label{eq:g-conv-local}
\|g^\ep - g\|_{L^2(A_\delta)} \longrightarrow 0\qquad\text{by Vitali's convergence theorem.}
\end{align*}

On the set when the density is close to zero, Lemma~\ref{lem:MV} and estimate similar to \eqref{smalness} imply that
\begin{align*}
\lim_{\delta\to 0} \limsup_{\ep\to 0}\iint_{A_\delta^c} |g^\ep - g|^2\,\dx\,\dt = 0,
\end{align*}
which concludes the proof of \eqref{eq:kin-dens-conv}.

\subsection{Convergence of $\tilde\rho^\ep\,W\ast\tilde\rho^\ep$}\label{ssec:conv_interW} 
We first notice that  
\bq\label{hm_conv}
\|\tilde\rho^\e_t - \rho_t\|_{L^1_2} \to 0,\qquad \|f\|_{L^1_\alpha} \coloneq \int_\R |x|^\alpha f\,\d x.
\eq
Indeed, for any $\eta \geq 1$, we observe
\[\begin{aligned}
\|\tilde\rho^\e - \rho\|_{L^1_2} &= \|\tilde\rho^\e - \rho\|_{L^1_2(B(0,\eta))} + \|\tilde\rho^\e - \rho\|_{L^1_2(\R \backslash B(0,\eta))}\cr
&\leq  \eta^2\|\tilde\rho^\e - \rho\|_{L^1(B(0,\eta))} + \frac{1}{\eta^\kappa}\lt(\|\tilde\rho^\e\|_{L^1_{2+\kappa}} + \|\rho\|_{L^1_{2+\kappa}}\rt).
\end{aligned}\]
Thus, taking the limit $\e \to 0$, and then sending $\eta \to \infty$ give \eqref{hm_conv}.

Based on the assumptions on $W$, we have that
\[
\lt|\int_\R W(x-y) f(y) \d y \rt| 
\leq c(1 + |x|^2)\|f\|_{L^1} + c\|f\|_{L^1_2}\qquad \forall\, f \in L^1_2(\R).
\]
Then, we estimate
\[\begin{aligned}
\|\tilde\rho^\e W * \tilde\rho^\e - \rho W * \rho\|_{L^1} 
&\leq \|\tilde\rho^\e (W * (\tilde\rho^\e - \rho))\|_{L^1} + \|(\tilde\rho^\e - \rho) (W * \rho)\|_{L^1}\cr
&\leq c\lt(\|\tilde\rho^\e - \rho\|_{L^1} + \|\tilde\rho^\e - \rho\|_{L^1_2}\rt)
\end{aligned}\]
for some $c > 0$ independent of $\e>0$. Using \eqref{hm_conv}, we then obtain
\[
\tilde\rho^\ep\,W\ast\tilde\rho^\ep \to \rho\, W\ast\rho\quad\text{strongly in $L^1(\R)$ for almost every $t\ge 0$}.
\]

\subsection{Convergence of $\tilde\rho^\ep\,\partial_x W\ast\tilde\rho^\ep$}

For any $\varphi \in L^\infty(0,T;L^\infty_{loc}(\R))$ and compact $K\subset \R$, we estimate
\[\begin{aligned}
&\int_0^T\!\! \int_K \varphi(t,x)\lt( \tilde\rho^\ep\,\partial_x W\ast\tilde\rho^\ep - \rho\,\partial_x W\ast\rho\rt) \dx \dt \cr
&\quad = \int_0^T\!\! \int_K \varphi\, (\tilde\rho^\ep - \rho)\, \pa_x W * \tilde\rho^\e\,\dx\dt + \int_0^T\!\! \int_K \varphi\, \rho\, \pa_x W * (\tilde\rho^\ep - \rho)\,\dx\dt =: I_\e + II_\e.
\end{aligned}\]
For $I_\ep$, we have
\[
|I_\e| \leq C(c_W, K, T) \|\tilde\rho^\e - \rho\|_{L^\infty(0,T;L^1 (K))}\|\varphi\|_{L^\infty} \sup_{0 \leq t \leq T}(\|\tilde\rho^\e(t)\|_{L^1} + \|\tilde\rho^\e(t)\|_{L^1_1})
\to 0
\]
as $\e \to 0$. As for $II_\e$, we obtain
\[\begin{aligned}
|II_\e| &\leq C(c_W, K,T)\|\rho\|_{L^1}\|\varphi\|_{L^\infty}\sup_{0 \leq t \leq T}(\|(\tilde\rho^\e- \rho)(t)\|_{L^1} + \|(\tilde\rho^\e - \rho)(t)\|_{L^1_1}) \to 0\cr
\end{aligned}\]
as $\e \to 0$. Together, this yields the convergence
\[
\tilde\rho^\ep\,\partial_x W\ast\tilde\rho^\ep \rightharpoonup \rho\,\partial_x W\ast\rho\quad\text{weakly in $L^1(0,T;L_{loc}^1(\R))$}.
\]

\subsection{Convergence of $\phi_{\ep}\ast \tilde\rho^\ep$}

Similarly, we consider a test function $\varphi \in L^1(0,T;L^1_{loc}(\R))$ and a compact set $K \subset \R$. Let us choose $\eta \geq 1$ such that $|K| \leq \eta$. Then, 
\[\begin{aligned}
&\lt|\int_0^T\!\! \int_K \int_\R \varphi(t,x) \phi(x-y)(\tilde\rho^\e - \rho)(t,y)\,\dy \dx \dt\rt|\cr
&\hspace{4em} = \lt|\int_0^T\!\! \int_K \lt(\int_{|y| \leq \eta + 1} + \int_{|y| \geq \eta + 1} \rt)\varphi(t,x) \phi(x-y)(\tilde\rho^\e - \rho)(t,y)\,\dy \dx \dt \rt| \cr
&\hspace{4em} \leq \|\tilde \rho^\e - \rho\|_{L^\infty((0,T) \times (-\eta - 1 , \eta + 1))}\|\phi 1_{|\cdot| \leq 2\eta + 1}\|_{L^1}\|\varphi\|_{L^1((0,T) \times K)} \cr
&\hspace{4em}\qquad + \frac{1}{(\eta + 1)^2} \|\phi 1_{|\cdot|\geq 1}\|_{L^\infty}\|\varphi\|_{L^1((0,T) \times K)} \sup_{0 \leq t \leq T} \Bigl(\|\tilde\rho^\ep(t)\|_{L^1_2} + \|\rho(t)\|_{L^1_2}\Bigr) \cr
&\hspace{4em} \leq C\|\tilde \rho^\e - \rho\|_{L^\infty((0,T) \times (-\eta - 1 , \eta + 1))} + \frac C{(\eta + 1)^2},
\end{aligned}\]
where $C>0$ is independent of $\e$ and $\eta$. Thus, by sending $\e \to 0$ first and then $\eta \to \infty$, we conclude the desired convergence, i.e.
\[
\phi\ast \tilde\rho^\ep \rightharpoonup^* \phi\ast \rho\quad\text{weakly-$*$ in $L^\infty(0,T;L^\infty_{loc}(\R))$}.
\]

\section{Recovery of the original system: existence of weak solutions}

To recover weak solutions to \eqref{main_eq}, we make use of the convergences provided in Section~\ref{sec:compactness} to first prove that the functional inequality obtain in Section~\ref{Sec:existence} (cf.\ \eqref{est_mod_app}) is retained when $\ep\to 0$ (cf.\ Theorem~\ref{thm:limit_energy}). We then conclude in the same result that the limit pair $(\rho,\sqrt{\rho}u)$ obtained in Section~\ref{sec:compactness} satisfies the weak formulation \eqref{2.3} and \eqref{2.4}. The main difficulty is that, while extension of $\rho^\ep$ to the whole space is possible,  for the gradient of $\rho^\ep$ it is not, due to the discontinuity at the boundary of $\Omega_\ep$. Therefore, we reformulate our functional inequalities in terms of measures that allow us to integrate over the whole space and easily pass to the limit in all the energy-type estimates.
We will use the following lower semicontinuity result w.r.t.\ distributional convergence of measures for integral functionals with lower semicontinuous and convex integrands. 

A sequence of measures $(\nu_n)_{n\ge 1}\in \calM_{loc}(\Omega)$ on Borel set $\Omega\subset\R^d$ is said to converge distributionally to $\nu\in\calM_{loc}(\Omega)$ if
\[
\int_\Omega \psi(x)\,\nu_n(\dx) \longrightarrow \int_\Omega \psi(x)\, \nu(\dx)\qquad \forall\,\psi\in C_c^\infty(\Omega).
\]
In this case, we simply write $\nu_n\rightharpoonup^* \nu$ in $\calM_{loc}(\Omega)$.

\medskip

The following is a corollary of a classical result in the Calculus of Variations (see eg.\ \cite{AFP00}).

\begin{proposition}\label{prop:lsc}
Let $g:\R\to [0,+\infty]$ be a superlinear, convex, and lower semicontinuous function. Then, the functional
\[
\mathcal{M}^+(\R)\times \mathcal{M}(\R)\ni (\nu,\eta)\mapsto\mathcal{G}(\nu,\eta) \coloneq \begin{cases}\displaystyle
	\int_\R g\biggl(\frac{\d \eta}{\d \nu}\biggr) \d\nu &\displaystyle\text{if $\eta\ll \nu$},\\
	+\infty &\text{otherwise,}
\end{cases}
\]
is lower semicontinuous w.r.t.\ the distributional convergence of measures.
\end{proposition}

In the following, we define the measures
\begin{gather*}
\sigma_t^\ep(\dx)\coloneq\tilde \rho^\ep(t,x)\d x,\qquad
\Sigma^\ep(\dt\,\dx)\coloneq \sigma_t^\ep(\dx)\d t,\qquad \omega_t^\ep(\dx)\coloneq \tilde\rho^\ep(t,x)u^\ep(t,x)\d x\\
j_t^\ep(\dx) \coloneq \partial_x\delta\calF_\ep(\rho^\ep(t,x))\, \sigma_t^\ep(\dx),\qquad 
J^\ep(\dt\dx)\coloneq j_t^\ep(\dx)\d t,\\
\vartheta_1^\ep(\dt\dx)\coloneq \mu(\tilde\rho^\ep(t,x))\d t\d x,\qquad \Upsilon_1^\ep(\dt\,\dx)\coloneq \partial_x u^\ep(t,x)\,\vartheta_1^\ep(\dt\,\dx),\\
\vartheta_2^\ep(\dt\dx)\coloneq \ep\alpha(\tilde\rho^\ep(t,x))^\alpha\d t\,\d x,\qquad \Upsilon_2^\ep(\dt\dx)\coloneq \partial_x u^\ep(t,x)\,\vartheta_2^\ep(\dt\,\dx),\\
\Pi^\ep(\dt\dx\dy) \coloneq \phi_\ep(x-y)\tilde\rho^\ep(t,x)\,\tilde\rho^\ep(t,y)\d t\,\d x\d y,\\
\varGamma_\ep(\dt\,\dx\,\dy)\coloneq \bigl[ u^\ep(t,x) - u^\ep(t,y)\bigr]\,\Pi^\ep(\dt\,\dx\,\dy).
\end{gather*}
Using the measures defined above, we then have that
\begin{align*}
    \calJ_{\tau,\lambda,\ep}(\rho^\ep,u^\ep) &= \frac{1}{2}\int_\R \biggl|\frac{\d(\omega^\ep + j^\ep)}{\d \sigma^\ep}\biggr|^2\!\!\d \sigma^\ep + \frac{\lambda}{2}\int_\R \biggl|\frac{\d\omega^\ep }{\d \sigma^\ep}\biggr|^2\!\!\d \sigma^\ep + (1+\lambda+\tau)
\widehat\calF_\ep(\sigma^\ep), \\
\calD_{\tau,\lambda,\ep}(\rho^\ep,u^\ep) &\ge  \int_\R \left[\frac{1}{2}\biggl|\frac{\d J^\ep}{\d \Sigma^\ep}\biggr|^2 \!\!+  \ell_{\tau,\lambda}\biggl|\frac{\d \omega_t^\ep}{\d \sigma_t^\ep}\biggr|^2\right] \!\!\d \Sigma^\ep 
+ \lambda \int_\R \biggl|\frac{\d \Upsilon_1^\ep}{\d \vartheta_1^\ep}\biggr|^2 \!\!\d \vartheta_1^\ep + c_{\lambda} \iint_{\R \times \R} \biggl|\frac{\d \varGamma^\ep}{\d \Pi^\ep}\biggr|^2 \!\!\d \Pi^\ep,
\end{align*}
where we simply dropped the $(\vartheta_2^\ep,\Upsilon_2^\ep)$ term, and we set
\[
\calM^+(\R)\ni \sigma\mapsto \widehat\calF_\ep(\sigma) \coloneq \begin{cases}\displaystyle
\int_\R U_\ep(\rho(x))\d x + \frac{1}{2}\int_\R W\ast \sigma\,\d \sigma &\displaystyle\text{if $\sigma\ll \calL$ with $\rho=\frac{\d\sigma}{\d\calL}$}, \\
+\infty &\text{otherwise},
\end{cases}
\]
with $U_\ep(\rho) = \rho^\gamma/(\gamma-1) + \ep^\beta \rho^\alpha/(\alpha-1)$. 

\medskip

Our first result pertains to the liminf inequality of the free energy term $\widehat\calF_\ep$.

\begin{lemma}\label{lem:free-energy-lsc}
Suppose $\ep^\beta R_\ep^{1-\alpha}\to 0$. Then,
\[
\widehat \calF(\sigma) \le \liminf_{\ep\to 0} \widehat \calF(\sigma^\ep),
\]
where
\[
\calM^+(\R)\ni \sigma\mapsto \widehat\calF(\sigma) \coloneq \begin{cases}\displaystyle
	\int_\R U(\rho(x))\d x + \frac{1}{2}\int_\R W\ast \sigma\,\d \sigma &\displaystyle\text{if $\sigma\ll \calL$ with $\rho=\frac{\d\sigma}{\d\calL}$}, \\
	+\infty &\text{otherwise}.
\end{cases}
\]
\end{lemma}
\begin{proof}
Since $[0,+\infty)\ni \rho\mapsto U(\rho)=\rho^\gamma/(\gamma-1)$ is superlinear, convex, and lower semicontinuous, the functional
\[
\calM^+(\R)\ni \sigma\mapsto \begin{cases}\displaystyle
	\int_\R U(\rho(x))\d x &\displaystyle\text{if $\sigma\ll \calL$ with $\rho=\frac{\d\sigma}{\d\calL}$}, \\
	+\infty &\text{otherwise}.
\end{cases}
\]
is lower semicontinuous w.r.t.\ the distributional convergence of measures due to Proposition~\ref{prop:lsc}, with $g=U$. On the other hand, we have that
\[
\ep^\beta \intM{(\rho^\ep)^\alpha} \le \ep^\beta\,(2R_\ep)^{1-\alpha}\longrightarrow 0\qquad\text{since\; $\ep^\beta R_\ep^{1-\alpha}\to 0$.}
\]
Due to the convergence estimate in Section \ref{ssec:conv_interW}, for the interaction energy, we have
\[
\lim_{\e \to 0}\int_\R W\ast \sigma^\ep\d \sigma^\ep = \int_\R W\ast\sigma\d\sigma.
\]
Altogether, we obtain the asserted liminf inequality.
\end{proof}

\begin{lemma}\label{lem:measure-conv}
Under the convergences given in \eqref{eq:convergences-lsc}, we have for some (not relabelled) subsequence:
\begin{gather*}
	j_t^{\ep} \rightharpoonup^* j_t(\dx)\coloneq D \rho^{\gamma}(\dx) + \rho\,\partial_x W\ast \rho\,\dx\quad\text{in $\calM_{loc}(\R)$ for every $t\ge 0$},\\
	J^{\ep}\rightharpoonup^* J(\dt\,\dx)\coloneq j_t(\dx)\d t\quad\text{in $\calM_{loc}((0,T)\times \R)$},\\
	\Upsilon_1^{\ep}\rightharpoonup^* \Upsilon,\quad \Upsilon_2^{\ep}\rightharpoonup^* 0\quad \text{in $\calM_{loc}((0,T)\times \R)$},\\
	\Pi^{\ep}\rightharpoonup \Pi(\dt\,\dx\,\dy)\coloneq \phi(x-y)\rho(t,x)\,\rho(t,y)\d t\d x\d y\quad\text{in $\calM((0,T)\times \R\times \R)$},\\
	\varGamma^{\ep}\rightharpoonup \varGamma(\dt\,\dx\,\dy)\coloneq \bigl[ u(t,x) - u(t,y)\bigr]\,\Pi(\dt\,\dx\,\dy)\quad\text{in $\calM((0,T)\times \R\times\R)$},
\end{gather*}
where, for all $\psi\in C_c^\infty((0,T)\times\R)$, $\Upsilon$ can be expressed as
\[\int_0^T\!\!\!\int_\R  \psi \d \Upsilon(\dx\,\dt) = -\gamma\int_0^T\!\!\!\int_\R  \partial_x\psi \,\rho^{\gamma}\,u\,\dx\,\dt - \frac{\gamma}{\gamma-\sfrac12}\int_0^T\!\!\!\int_\R \psi \,{\sqrt{\rho} u}\,\partial_x \rho^{\gamma-\sfrac{1}{2}}\,\dx\,\dt .
\]
\end{lemma}
\begin{proof}
\begin{enumerate}
	\item
	For any $\psi\in C_c^\infty(\R)$ and almost every $t\ge 0$, we have for $\ep\ll 1$ sufficiently small
	\begin{align*}
		\int_\R \psi\d j_t^\ep &= \int_{\Omega_\ep} \psi\,\rho^\ep\, \partial_x \varphi_\ep(\rho^\ep) \d x + \int_\R \psi \,\tilde\rho^\ep\,\partial_x W\ast \tilde\rho^\ep \d x \\
		&= -\int_\R \partial_x\psi \bigl[ (\tilde\rho^\ep)^\gamma + \ep(\tilde\rho^\ep)^\alpha\bigr]\d x + \int_\R \psi \,\tilde\rho^\ep\,\partial_x W\ast \tilde\rho^\ep \d x \\
		&\qquad\longrightarrow -\int_\R \partial_x\psi\, \rho^\gamma\d x + \int_\R \psi\, \rho\,\partial_x W\ast\rho\d x.
	\end{align*}
	Since $\rho^\gamma(t)\in BV(\R)$ for every $t\ge 0$, this implies that $j_t\in \calM(\R)$, and consequently, $j_t^\ep \rightharpoonup^* j_t$ in $\calM_{loc}(\R)$ for every $t\ge 0$.
	
	\item This follows easily from (1).
	\item We first establish that $(\Upsilon_1^\ep)_{\ep>0}$ admits a subsequence that converges in $\calM_{loc}((0,T)\times\R)$. This follows from the fact that
	\[
	\sup_{\ep>0}\|\Upsilon_1^\ep\|_{TV} \le \sup_{\ep>0}\|\tilde\rho^\ep\|_{L^\gamma}^\gamma\biggl\|\frac{\d \Upsilon_1^\ep}{\d \vartheta_1^\ep}\biggr\|_{L^2(\vartheta_1^\ep)} <+\infty,
	\]
	which provides an $\Upsilon\in\calM((0,T)\times\R)$ and a (not relabelled) subsequence such that $\Upsilon_1^\ep \rightharpoonup^* \Upsilon$ in $\calM_{loc}((0,T)\times\R)$. On the other hand, we have for any $\psi\in C_c^\infty((0,T)\times\R)$,
	\begin{align*}
		\int_0^T\!\!\!\int_\R \psi\,\Upsilon_1^\ep(\dt\,\dx) &= \int_0^T\!\!\!\int_\R \psi \,\mu(\tilde\rho^\ep)\,\partial_x u^\ep\d x\d t \\
		&= - \gamma\int_0^T\!\!\!\int_\R \partial_x\psi \,(\tilde\rho^\ep)^{\gamma}\, u^\ep\d x\d t 
		- \frac{\gamma}{\gamma-\sfrac12}\int_0^T\!\!\!\int_\R  \psi\,\sqrt{\tilde\rho^\ep}u^\ep\,\partial_x(\rho^\ep)^{\gamma-\sfrac12}\d x \d t.
	\end{align*}
	The first term converges by means of the dominated convergence theorem. As for the second term, we first note that since
	\[
	\sup_{\ep>0}\int_0^T\!\!\!\int_\R \mathbbm{1}_{\Omega_\ep}|\partial_x(\rho^\ep)^{\gamma-\sfrac12}|^2\,\dx\,\dt <+\infty,
	\]
	there exists some $\zeta\in L^2((0,T)\times\R)$ and some (not relabelled) subsequence, for which
	\[
	\mathbbm{1}_{\Omega_\ep}\partial_x(\rho^\ep)^{\gamma-\sfrac12} \rightharpoonup \zeta\quad\text{weakly in $L^2((0,T)\times\R)$.}
	\]
	On the other hand, we have for any $\psi\in C_c^\infty((0,T)\times\R)$ and $\ep\ll 1$ sufficiently small,
	\begin{align*}
		\int_0^T\!\!\!\int_{\Omega_\ep} \psi\,\partial_x(\rho^\ep)^{\gamma-\sfrac12}\,\dx\,\dt = -\int_0^T\!\!\!\int_\R\partial_x\psi \,(\rho^\ep)^{\gamma-\sfrac12}\,\dx\,\dt \longrightarrow -\int_0^T\!\!\!\int_\R\partial_x\psi\,\rho^{\gamma-\sfrac12}\,\dx\,\dt,
	\end{align*}
	i.e.\ we can identify $\zeta = \partial_x \rho^{\gamma-\sfrac12}$. Consequently, second term converges as a product of a strongly converging sequence and a weakly converging sequence. Together, this implies that for any $\psi\in C_c^\infty((0,T)\times\R)$,
	\begin{align*}
		\int_0^T\!\!\!\int_\R \psi\,\Upsilon(\dt\,\dx) &= -\int_0^T\!\!\!\int_\R \partial_x\psi \,\mu(\rho)\,u\d x\d t - \frac{\gamma}{\gamma-\sfrac12}\int_0^T\!\!\!\int_\R\psi \sqrt{\rho}u\,\partial_x \rho^{\gamma-\sfrac12}\d x\d t.
	\end{align*}
	
	As for $(\Upsilon_2^\ep)_{\ep>0}$, we first notice that for any compact set $A\Subset (0,T)\times\R$:
	\[
	|\vartheta_2^\ep|(A) = \ep^\beta \alpha\int_A (\tilde\rho^\ep)^\alpha\d t\d x \le \ep^\beta \alpha \|\tilde\rho^\ep\|_{L^\infty}^\alpha |A|\longrightarrow 0.
	\]
	Consequently,
	\[
	|\Upsilon_2^\ep|(A) \le |\vartheta_2^\ep|(A) \biggl\|\frac{\d \Upsilon_2^\ep}{\d \vartheta_2^\ep}\biggr\|_{L^2(\vartheta_2^\ep)} \longrightarrow 0\qquad\forall\,\text{compact $A\Subset (0,T)\times\R$},
	\]
	thus implying that $\Upsilon_2^\ep\rightharpoonup^* 0$ in $\calM_{loc}((0,T)\times\R)$ as required.
	\item Since $\phi\ast \tilde\rho^\ep\rightharpoonup^* \phi\ast\rho$ weakly-$*$ in $L^\infty((0,T)\times\R)$ and $\tilde\rho^\ep \to \rho$ strongly in $L_{loc}^1((0,T)\times\R)$, we easily deduce for every $\psi\in C_c((0,T)\times\R\times\R)$:
	\[
	\int_0^T\!\!\!\int_\R \psi\d \Pi^\ep = \int_0^T\!\!\!\int_\R \psi\,\tilde\rho^\ep\,\phi\ast\tilde\rho^\ep \d x\d t \longrightarrow \int_0^T\!\!\!\int_\R \psi\,\rho\,\phi\ast \rho \d x\d t = \int_0^T\!\!\!\int_\R \psi\d \Pi.
	\]
	\item For any $\psi(t,x,y) = \psi_1(t,x)\,\psi_2(t,y)$ with $\psi_1,\psi_2\in C_b^\infty((0,T)\times\R)$, we have that
	\begin{align*}
		\int_0^T\!\!\!\iint_{\R\times\R} \psi \d \varGamma^\ep &= \int_0^T\!\!\!\iint_{\R\times\R} \psi(t,x,y)\phi(x-y)\bigl[\tilde u^\ep(t,x)-\tilde u^\ep(t,y)\bigr]\,\tilde\rho^\ep(t,x)\,\tilde\rho^\ep(t,y)\d x\d y\d t \\
		&= \int_0^T\!\!\!\int_\R \psi_1\,\tilde\rho^\ep u^\ep\,\phi\ast(\psi_2\tilde\rho^\ep)\d x\d t - \int_0^T\!\!\!\int_\R \psi_2\,\tilde\rho^\ep u^\ep\,\phi\ast(\psi_1\tilde\rho^\ep)\d x\d t.
	\end{align*}
	Owing to the convergence $\tilde\rho^\ep u^\ep \rightharpoonup \rho u$ in $L^1((0,T)\times\R)$, we obtain as in (4),
	\begin{align*}
		\int_0^T\!\!\!\iint_{\R\times\R} \psi \d \varGamma^\ep \longrightarrow &\int_0^T\!\!\!\int_\R  \psi_1\,\rho u \,\phi\ast(\psi_2\rho)\d x\d t - \int_0^T\!\!\!\int_\R\psi_2\,\rho u\,\phi\ast(\psi_1\rho)\d x\d t \\
		&\qquad = \int_0^T\!\!\!\iint_{\R\times\R}  \psi(t,x,y)\, \bigl[u(t,x)-u(t,y)\bigr]\, \Pi(\dt\,\dx\,\dy) = \int_0^T\!\!\!\iint_{\R\times\R} \psi \d \varGamma.
	\end{align*}
	Since any function in $C_b((0,T)\times\R)$ can be approximated by functions of the form $\psi$ above, we conclude that $\varGamma^\ep\rightharpoonup\varGamma$ in $\calM((0,T)\times\R\times\R)$ as asserted. \qedhere
\end{enumerate}
\end{proof}

\begin{theorem}\label{thm:limit_energy}
The results of Lemma~\ref{lem:free-energy-lsc}, Proposition~\ref{prop:lsc} and Lemma~\ref{lem:measure-conv} yield
\equ{\label{emM-2}
    \calJ_{\tau,\lambda}(\rho(t),u(t)) + \int_0^t \calD_{\tau,\lambda}(\rho,u)\,\dt \le \calJ_{\tau,\lambda}(\rho_0,u_0).
}
In particular, the pair $(\rho,\sqrt{\rho}u)$ is a weak solution of \eqref{main_eq}.
\end{theorem}
\begin{proof}
We begin by recalling the integrated BD-type inequality \eqref{est:BD_app}
\begin{align*}
	&\calJ_{\tau,\lambda,\ep}(\rho^\ep(t),u^\ep(t)) +\left(\frac{1}{2}-a_\ep\right)\int_0^t\!\!\int_\R\biggl|\frac{\d J^\ep}{\d \Sigma^\ep}\biggr|^2  \!\!\d \Sigma^\ep + \lambda \int_0^t\!\!\int_\R \biggl|\frac{\d \Upsilon_1^\ep}{\d \vartheta_1^\ep}\biggr|^2 \!\!\d \vartheta_1^\ep + c_{\lambda} \int_0^t\!\!\iint_{\R \times \R} \biggl|\frac{\d \varGamma^\ep}{\d \Pi^\ep}\biggr|^2 \!\!\d \Pi^\ep\\
 &\hspace{10em}\le - \left(\ell_{\tau,\lambda}-a_\ep\right)\int_0^t\!\!\int_\R \tilde\rho^\ep|u^\ep|^2\, \dx\dt + \calJ_{\tau,\lambda,\ep}(\rho^\ep_0,u^\ep_0)
\end{align*}
for $\ep\ll1$ sufficiently small such that $a_\ep <1/2$. For the terms on the left-hand side, we simply apply Proposition~\ref{prop:lsc} under the convergences provided in Lemma~\ref{lem:measure-conv}. As for the right-hand side, we make use of the strong convergence provided by \eqref{eq:convergences-lsc} for the first term and the convergence of the initial data in Proposition~\ref{prop_convJ} for the second term. Together, we obtain \eqref{emM-2}.

Moreover, due to the lower semicontinuity result, we also find that
\begin{align*}
	\int_\R \biggl|\frac{\d D\rho^\gamma}{\d \sigma_t}\biggr|^2\d \sigma_t <+\infty\qquad\text{for every $t\ge 0$},
\end{align*}
thus implying that $D\rho^\gamma \ll \sigma_t$ for every $t\ge 0$. Since $\sigma_t = \rho_t\,\dx$,  $D\rho_t^\gamma$ is absolutely continuous w.r.t.\ to the Lebesgue measure with $\d D\rho_t^\gamma/\dx = \partial_x\rho_t^\gamma$, where $\partial_x\rho_t^\gamma$ is the weak derivative of $\rho_t^\gamma$.

The fact that $(\rho,m)$ satisfies \eqref{2.3} and \eqref{2.4} follow from the asserted convergences and that 
\begin{align*}
	&\int_0^T\!\!\!\int_\R \mu_\ep(\tilde\rho^\ep)\,\partial_x \psi\,\partial_x u^\ep\d x\d t = \int_0^T\!\!\!\int_\R  \partial_x \psi\,\Upsilon_1^\ep(\dt\,\dx) + \int_0^T\!\!\!\int_\R  \partial_x \psi\,\Upsilon_2^\ep(\dt\,\dx) \\
	&\hspace{4em}\longrightarrow \int_0^T\!\!\!\int_\R  \partial_x \psi\,\Upsilon(\dt\,\dx) = -\int_0^T\!\!\!\int_\R \partial_x\psi \,\mu(\rho)\,u\d x\d t - \frac{\gamma}{\gamma-\sfrac12}\int_0^T\!\!\!\int_\R \psi \sqrt{\rho}u\,\partial_x \rho^{\gamma-\sfrac12}\d x\d t,
\end{align*}
for all $\psi\in C_c^\infty((0,T)\times\R)$, therewith concluding the proof.
\end{proof}

\section{Long-time behavior of solutions}\label{Sec:lt}

This section is devoted to the proof of Theorem~\ref{Th:2}. We first obtain from \eqref{est:BD_app_uniform} that 
\begin{gather}
\int_{\Om_\e}{\rho^\e (\pa_x \delta \calF_\e(\rho^\e))^2}\d x, \quad \int_{\Om_\e}{ \mu_\e(\rho^\e)(\pa_x u^\e)^2} \d x \in L^1(0,\infty), \label{lt_est1}\\
\calK_{\phi^\ep}(\rho^\ep,u^\ep) = \iint_{\Om_\e \times \Om_\e}{ \phi_\e(x-y)(u^\e(x) - u^\e(y))^2 \rho^\e(x) \rho^\e(y)}\d x \d y \in L^1(0,\infty),\label{lt_est2}\\
\int_{\Om_\e}{ \rho^\e (u^\e)^2}\d x \in L^1(0,\infty).\label{lt_est3}
\end{gather}
We then claim that 
\bq\label{claim}
\frac{\d}{\d t}\int_{\Om_\e}{\rho^\e |u^\e + \pa_x \delta \calF_\e(\rho^\e)|^2} \d x ,\qquad \frac{\d}{\d t} \int_{\Om_\e}{ \rho^\e |u^\e|^2} \d x \in L^1(0,\infty).
\eq
For this, we revisit the estimates in the proof of Lemma \ref{lem:bd-estimate}. Note that 
\[
\begin{aligned}
	&\frac12\frac{\d}{\d t}\int_{\Om_\e}{\rho^\e |u^\e + \pa_x \delta \calF_\e(\rho^\e)|^2} \d x\cr
	&\quad =\frac12 \frac{\d}{\d t}\int_{\Om_\e}{\rho^\e |u^\e|^2} \d x + \frac12\frac{\d}{\d t}\int_{\Om_\e}{\rho^\e (\pa_x \delta \calF_\e(\rho^\e))^2} \d x + \frac{\d}{\d t}\int_{\Om_\e}{\rho^\e u^\e \pa_x \delta \calF_\e(\rho^\e)} \d x\cr
	&\quad =: I_1 + I_2 + I_3,
\end{aligned}
\]
where
\[\begin{aligned}
	I_1&= - \int_{\Om_\e}{\rho^\e u^\e \pa_x \delta \calF_\e(\rho^\e)} \d x - \int_{\Om_\e}{\mu_\e(\rho^\e) (\pa_x u^\e)^2} \d x \cr
	&\quad - \frac12  \iint_{\Om_\e \times \Om_\e}{ \phi_\e(x-y)(u^\e(x) - u^\e(y))^2 \rho^\e(x)\rho^\e(y)} \d x \d y  - \tau\int_{\Om_\e}{ \rho^\e (u^\e)^2}\d x
\end{aligned}\]
and thus
\[\begin{aligned}
	|I_1| &\leq \left(\frac12+\tau\right)\int_{\Om_\e}{ \rho^\e (u^\e)^2}\d x  + \frac12\int_{\Om_\e}{\rho^\e (\pa_x \delta \calF_\e(\rho^\e))^2} \d x + \int_{\Om_\e}{\mu_\e(\rho^\e) (\pa_x u^\e)^2} \d x \cr
	&\qquad+ \frac12 \iint_{\Om_\e \times \Om_\e}{ \phi_\e(x-y)(u^\e(x) - u^\e(y))^2 \rho^\e(x)\rho^\e(y)} \d x \d y \quad \in L^1(0,\infty),
\end{aligned}\]
due to \eqref{lt_est1}, \eqref{lt_est2}, and \eqref{lt_est3}. 

For the other terms, we recall from the proof of Lemma \ref{lem:bd-estimate} that 
\[\begin{aligned}
	I_2 + I_3 =  
	& - \int_{\Om_\e}{ \rho^\e u^\e[\pa_{xx}W *^\e, u^\e] \rho^\e}\d x+ \int_{\Om_\e}{\mu_\e( \rho^\e)(\pa_x u^\e)^2} \d x   -\tau\int_{\Om_\e}{\rho^\e u^\e\px\delta\calF_\e(\rho^\e)}\d x\cr
 =  &  - \int_{\Om_\e} {\rho^\e (\pa_x \delta \calF_\e(\rho^\e))^2} \d x+ \int_{\Om_\e}{\rho^\e [(\phi_\e - \pa_{xx}\tw)*^\e, u^\e] \rho^\e \pa_x \delta \calF_\e(\rho^\e)}\d x \\ 
	&   - \int_{\Om_\e}{ \rho^\e u^\e[\pa_{xx}\tw *^\e, u^\e] \rho^\e}\d x+ \int_{\Om_\e}{\mu_\e( \rho^\e)(\pa_x u^\e)^2} \d x   -\tau\int_{\Om_\e}{\rho^\e u^\e\px\delta\calF_\e(\rho^\e)}\d x.
\end{aligned}\]
Thus, we deduce
\[\begin{aligned}
	|I_2 + I_3| 
 &\leq c \int_{\Om_\e}{ \rho^\e (\pa_x \delta \calF_\e(\rho^\e))^2}\d x \cr
 &\quad + c \|\phi_\e * \rho^\e\|_{L^\infty}\iint_{\Om_\e \times \Om_\e}{\phi_\e(x-y)(u^\e(x) - u^\e(y))^2 \rho^\e(x)\rho^\e(y)} \d y \d x\cr
	&\quad + c(1 + \|\phi_\e * \rho^\e\|_{L^\infty} + \|\phi * \rho^\e\|_{L^\infty} ) \int_{\Om_\e}{ \rho^\e |u^\e|^2}\d x + \int_{\Om_\e}{\mu_\e( \rho^\e)(\pa_x u^\e)^2}\d x
\end{aligned}\]
for some $c>0$ independent of $t$ and $\e$, where we used the estimate
\begin{align*}
\int_{\Om_\e}{ \rho^\e |[\pa_{xx}\tw *^\e, u^\e] \rho^\e}|^2\d x 
\leq c(1 + \|\phi *^\e \rho^\ep\|_{L^\infty})^2  \int_{\Om_\e}{ \rho^\e |u^\e|^2}\d x.
\end{align*}
Hence, we have also that $I_2+I_3 \in L^1(0,\infty)$. Combining all of the above estimates shows \eqref{claim}. On the other hand, by \eqref{lt_est1} and \eqref{lt_est3}, we obtain
\[
\int_{\Om_\e}{ \rho^\e |u^\e + \pa_x \delta \calF_\e(\rho^\e)|^2}\d x \in L^1(0,\infty),
\]
and subsequently, this, together with \eqref{lt_est3} and \eqref{claim}, yields
\[
\int_{\Om_\e}{ \rho^\e |u^\e + \pa_x \delta \calF_\e(\rho^\e)|^2} \d x,\ \int_{\Om_\e}{  \rho^\e |u^\e|^2}\d x \in W^{1,1}(0,\infty).
\]
This shows 
\[
\int_{\Om_\e}{ \rho^\e |u^\e + \pa_x \delta \calF_\e(\rho^\e)|^2}\d x \to 0, \qquad   \int_{\Om_\e}{  \rho^\e |u^\e|^2} \d x\to 0, 
\]
due to \cite{KR60}, and consequently also
\[
\int_{\Om_\e}{ \rho^\e ( \pa_x \delta \calF_\e(\rho^\e))^2}\d x \to 0\qquad\text{as $t\to\infty$}.
\]
Here, we would like to remark that the above estimates in $W^{1,1}(0,\infty)$ are independent of $\e>0$.

On the other hand, by using almost the same argument as in the proof of Lemma \ref{lem:free-energy-lsc}, we observe
\begin{align}
\intO{\rho\,( \pa_x \delta \calF(\rho))^2} &\leq  \liminf_{\e \to 0}\int_{\Om_\e}{ \rho^\e ( \pa_x \delta \calF_\e(\rho^\e))^2}\d x \nonumber
\intO{ \rho |u|^2} &\leq  \liminf_{\e \to 0} \int_{\Om_\e}{  \rho^\e |u^\e|^2} \d x \nonumber 
\end{align}
where the right-hand sides tend to zero as $t \to \infty$. 

To get more information about the limits, we further assume that the interaction potential $W$ satisfies \eqref{add_con}. Then, by invoking the estimates in the proof of Corollary \ref{cor_bddr}, we deduce 
\[
\|\pa_x( \rho^{\gamma-\frac{1}{2}})(t) \|_{L^2} \leq C\quad \mbox{and} \quad \|\rho(t)\|_{L^\infty} \leq C
\]
for some $C>0$ independent of $t$. Moreover, we find
\[
\|\pa_x \rho^\gamma (t)\|_{L^2} \leq \frac{\gamma}{\gamma - \frac12} \|\rho(t)\|_{L^\infty}^\frac12 \|\pa_x( \rho^{\gamma-\frac{1}{2}})(t) \|_{L^2} \leq C.
\]
Thus, we obtain
\[
    \sup\nolimits_{t\in[0,\infty)}\|\rho^\gamma(t) \|_{H^1} \leq C
\]
from which we also find that there is a sequence $\{t_n\} \subset (0,\infty)$ with $t_n \to \infty$ as $n \to \infty$ such that 
\begin{align}\label{con_rhoga}
	\begin{aligned}
		\rho^\gamma(t_n) &\to  \nu  \quad \mbox{strongly in $L^2(\R)$ and a.e.,} \cr
		\pa_x (\rho^\gamma)(t_n) &\rightharpoonup \eta \quad \mbox{weakly in $L^2(\R)$}
	\end{aligned}
\end{align}
for some $ \nu  \in H^1(\R)$ and $\eta \in L^2(\R)$. Now we define $\rho_\infty$ by 
\[
\rho_\infty \coloneq  \nu ^\frac1\gamma.
\]
We first easily observe $\eta = \pa_x  \nu $ a.e. Indeed, for any $\varphi \in C^\infty_c(\R)$
\[
\intO{\varphi \, \eta }= \lim_{n \to \infty} \intO{\varphi \,\pa_x (\rho^\gamma)(t_n)} = - \lim_{n \to \infty} \intO{\pa_x \varphi \, \rho^\gamma(t_n)}  = - \intO{\pa_x \,\varphi  \nu  }.
\]
Now we claim that $\rho_\infty$ satisfies 
\[
\intO{\rho_\infty(x)\lr{\px\delta\calF(\rho_\infty(x))}^2}=0.
\]
For this, similarly as in the proof of Lemma \ref{lem:measure-conv}, it suffices to show 
\[
\bigl[\partial_x \rho^{\gamma}(t_n) + \rho(t_n)\,\partial_x W\ast \rho(t_n)\bigr]\d x = j_{t_n} \rightharpoonup^* j_\infty(\dx)\coloneq \bigl[\partial_x \rho_\infty^{\gamma} + \rho_\infty\,\partial_x W\ast \rho_\infty\bigr]\d x.
\]
Here due to \eqref{con_rhoga}, we readily see that
\[
\partial_x \rho^{\gamma}(t_n) \d x \rightharpoonup^* \partial_x \rho_\infty^{\gamma}  \d x.
\]
Since $\gamma > 1$, the convergence of $\rho^\gamma$ in \eqref{con_rhoga} implies $\rho(t_n) \to \rho_\infty$ in $L^1_{loc}(\R)$. More precisely, for any bounded set $A \subset \R$,
\[
\| \rho(t_n) - \rho_\infty\|_{L^1(A)} = \| (\rho(t_n)^\gamma)^\frac1\gamma -  \nu ^\frac1\gamma\|_{L^1(A)} \leq C(|A|, \gamma)(\|\rho^\gamma(t_n)\|_{L^2(A)} + \| \nu \|_{L^2(A)}).
\]
Together with the almost everywhere convergence $\rho(t_n) \to \rho_\infty$, we can then conclude the claim by the Lebesgue dominated convergence theorem. This yields
\[
\rho(t_n)\,\partial_x W\ast \rho(t_n) \d x \rightharpoonup^* \rho_\infty\,\partial_x W\ast \rho_\infty \d x.
\]
This completes the proof of Theorem~\ref{Th:2}.

\begin{remark}
    We notice from \eqref{x2rho} that the second moment of $\rho$ may not be uniformly bounded in time. Thus, the left-hand side of \eqref{est_potw} may not be uniformly bounded in time. The additional assumption \eqref{add_con} is made to guarantee that the estimate \eqref{est_potw} is independent of $t>0$, which consequently provides the uniform-in-time estimate required to deduce the convergences \eqref{con_rhoga}.
\end{remark}

 \appendix
 
 \section{Local existence and uniqueness of solutions to the approximate system}\label{AppendixA}
 
 In this appendix, we present the local-in-time existence and uniqueness of regular solutions to the approximate system. Let us first recall our approximate system:
 \begin{subequations}\label{app_eq_a}
	\begin{align}
		\pa_t \rho + \pa_x (\rho u) &= 0, \qquad (t,x) \in (0,T)\times\Om_\ep,\label{app_eq1_a}\\
		\pa_t (\rho u) + \pa_x (\rho u^2) &= - \rho \pa_x \delta\calF_\ep(\rho) + \pa_x (\mu_\ep(\rho)\pa_x u) + \rho [\phi_\ep *^\ep, u]\rho -\tau\rho u, \label{app_eq2_a}
	\end{align}
\end{subequations}
with the initial data 
\[
(\rho, u)(0,x) = (\rho_0^\e, u_0^\e)(x), \quad x \in \Om_\e,
\]
and the boundary condition
\[
u = 0, \quad x \in \partial\Om_\e.
\]
For simplicity, we omit the $\e$-dependence in the solutions $(\rho, u)$ in \eqref{app_eq_a}. 

Since the local existence theory for compressible Navier--Stokes system in the absence of vacuum is by now well established, we briefly sketch the idea of the proof for the local well-posedness theory for the system \eqref{app_eq_a}. We would like to remark that our strategy is similar to that of \cite{CDNP20}, where the existence theory for one-dimensional compressible fluid models with periodic boundary conditions is investigated. 

\bigskip

\paragraph{\bf Step I} We first construct a sequence of approximate solutions $\{(\rho_n, u_n)\}_{n\geq0}$ to the following system:
\begin{subequations}\label{ru_n}
\begin{align}
&\pa_t \rho_{n} + u_{n-1} \pa_x \rho_{n} = - \rho_{n-1} \pa_x u_{n-1}, \quad (t,x) \in \R_+ \times \Om_\e, \label{ru_n_1}\\
&\partial_t u_n - \frac{\mu_\e(\rho_{n})}{\rho_{n}} \partial_{x}^2 u_n = G(\rho_{n-1}, u_{n-1}) \label{ru_n_2}
\end{align}
\end{subequations}
subject to the initial data, first iteration step, and boundary condition:
\[
(\rho_n(0,x), u_n(0,x)) = (\rho_0^\e, u_0^\e)(x), \quad   n \in \mathbb{N}, \quad x \in \Om_\e,
\]
\[
(\rho_0(t,x), u_0(t,x)) = (\rho_0^\e, u_0^\e)(x), \quad (t,x)  \in \R_+ \times \Om_\e,
\]
and
\[
u_n(t,x)=0, \quad   n \in \mathbb{N}\cup\{0\}, \quad (t,x)  \in \R_+ \times \partial\Om_\e.
\]
Here $G(\rho_{n-1}, u_{n-1})$ is given by
\[
G(\rho_{n-1}, u_{n-1}) = u_{n-1} \partial_x u_{n-1} -  \partial_x \delta\calF_\e(\rho_{n-1}) + \frac{\partial_x \mu_\e(\rho_{n-1})}{\rho_{n-1}} \partial_x u_{n-1} +  [\phi_\e *^\e, u_{n-1}]\rho_{n-1}.
\]

We first solve \eqref{ru_n_1} with $n=1$ to get a solution $\rho_1$, and then solve \eqref{ru_n_2} with the obtained solution $\rho_1$. Due to the smoothness of initial data $(\rho_0^\e, u_0^\e)$ together with the lower bound condition on $\rho_0^\e$, we have the global-in-time existence and uniqueness of smooth solutions $\{(\rho_n, u_n)\}_{n\geq0}$ to the system \eqref{ru_n} by the classical existence theory for transport and parabolic equations.

\bigskip

\paragraph{\bf Step II} Next, we provide uniform-in-$n$ bound estimates of the approximate solutions in our desired solution space. Precisely, let us assume that $(\rho_n, u_n) \in  \mathcal{X}_2(T)\times \mathcal{Y}_2(T)$ as given in \eqref{def:spaces} for some $T > 0$ with 
\[
\|(\rho_n, u_n)\|_{\calX_2(T) \times \calY_2(T)} \leq M,
\]
where $M > 0$ depends only on the initial data $(\rho_0^\e, u_0^\e)$, $\e$, $T$, and $\underline{\rho}^\e$. Then by employing the characteristic method, we deduce
\[
\min_{(t,x) \in [0,T_1] \times \overline\Om_\e} \rho_{n+1}(t,x) \geq \frac12 \underline\rho^\e
\]
for some small $T_1 > 0$, which depends only on $M$ and $\underline{\rho}^\e$, and using this lower bound estimate, we can obtain
\[
\|(\rho_{n+1}, u_{n+1})\|_{\calX_2(T_0) \times \calY_2(T_0)} \leq M.
\]
for sufficiently small $T_0 \in (0, T_1]$. Since $T_1 > 0$ is independent of $n$, we conclude by an inductive argument that
\[
\sup_{n \in \mathbb{N}}\|(\rho_n, u_n)\|_{\calX_2(T_0) \times \calY_2(T_0)} \leq M.
\]
Here, we would like to point out that in \cite{CDNP20}, the uniform-in-$n$ estimates are only obtained for $H^k$, $k \geq 2$, norms of approximate solutions. However, in our case, the homogeneous boundary condition on the velocity imposes us to estimate the time derivatives of solutions as well.

\bigskip

\paragraph{\bf Step III} In order to pass $n \to \infty$, we show that the sequence $\{(\rho_n, u_n)\}_{n\geq0}$ forms a Cauchy sequence in $ {L^\infty(0,T_0; L^2(\Om_\e))} \times {L^\infty(0,T_0; L^2(\Om_\e))}$. Indeed, if we define 
\[
\Delta \rho_n \coloneq \rho_{n+1}- \rho_n, \qquad \Delta u_n \coloneq u_{n+1}- u_n,
\]
then $\Delta \rho_n$ and $\Delta u_n$ satisfy
\begin{align*}
&\pa_t \Delta \rho_{n} + u_{n-1} \pa_x \Delta \rho_{n} = - \Delta u_{n-1} \pa_{x} \rho_{n} -\Delta \rho_{n-1} \pa_x u_{n} + \rho_{n-1} \pa_{x} \Delta u_{n-1}, \\
&	\pa_t \Delta u_n - \frac{\mu_\e(\rho_{n})}{\rho_{n}} \pa_{x}^2 \Delta u_n = - \left( \frac{\mu_\e(\rho_{n})}{\rho_{n}} - \frac{\mu_\e(\rho_{n-1})}{\rho_{n-1}} \right) \pa_{x}^2  u_{n-1}+ G(\rho_n, u_n)-G(\rho_{n-1}, u_{n-1})
\end{align*}
with
\[
(\Delta \rho_n, \Delta u_n)(0,x) = 0 \quad \mbox{and} \quad  \Delta u_n = 0 \ (t,x) \in (0,T_0) \times \partial \Om_\e.
\]
Then, we  use the uniform bound estimates obtained in (Step II) to deduce
\[
\frac{d}{dt}\lt(\|\Delta \rho_n(t)\|_{L^2}^2 + \|\Delta u_n(t)\|_{L^2}^2 \rt) \leq C \lt(\|\Delta \rho_n(t)\|_{L^2}^2 + \|\Delta \rho_{n-1}(t)\|_{L^2}^2 + \|\Delta u_n(t)\|_{L^2}^2 + \|\Delta u_{n-1}(t)\|_{L^2}^2 \rt),
\]
for $t \in (0,T_0)$, where $C>0$ is independent of $n$, from which we conclude the existence of a pair $(\rho, u) \in {L^\infty(0,T_0; L^2(\Om_\e))} \times {L^\infty(0,T_0; L^2(\Om_\e))}$ such that 
\[
(\rho_n, u_n) \to (\rho, u) \quad \mbox{strongly in } {L^\infty(0,T_0; L^2(\Om_\e))} \times {L^\infty(0,T_0; L^2(\Om_\e))}.
\]
Moreover, our uniform-in-n upper bound estimates imply
\begin{align*}
\begin{gathered}
	 \rho_n \rightharpoonup^* \rho  \text{ weakly-$*$  in }  L^\infty(0,T_0;H^2(\Om_\e)), \quad  \pa_{t} \rho_n \rightharpoonup^* \pa_{t}\rho  \text{ weakly-$*$ in }  L^\infty(0,T_0;H^1(\Om_\e)), \\
	 \rho_n \rightarrow \rho  \text{ strongly in }  L^2((0,T_0) \times \Om_\e), \quad  u_n \rightharpoonup^* u  \text{ weakly-$*$  in }  L^\infty(0,T_0;H^2(\Om_\e)),  \\
	 u_n \rightharpoonup u  \text{ weakly in }  L^2(0,T_0;H^3(\Om_\e)), \quad \pa_{t}u_n \rightharpoonup \pa_{t}u  \text{ weakly in }  L^2(0,T_0;H^1(\Om_\e)),  \quad \mbox{and} \\
	 u_n \rightarrow u \text{ strongly in }  L^2(0,T_0;H^1(\Om_\e))\cap L^\infty(0,T_0;L^2(\Om_\e)).
\end{gathered}
\end{align*}
Thus, the limit pair $(\rho, u)$ belongs to $ \mathcal{X}_2(T_0)\times \mathcal{Y}_2(T_0)$, concluding the proof of the existence part. 

\bigskip

\paragraph{\bf Step IV} We finally adapt the Cauchy estimates in (Step III) to obtain the uniqueness of the constructed solutions.

%
%
%
%
%
%

\section{Proof of Proposition \ref{prop_convJ}: approximation of the initial data}\label{app:initial}

In this part, we provide the details of proof for Proposition \ref{prop_convJ}. We begin with some quantitative estimates on $R_\e$ defined as in \eqref{choiceR_ep} that will be crucially used for the convergence estimates on approximations of initial data:
\begin{equation}\label{conv_Re_0}
\ep^{\frac{3}{4 (\gamma-1)}} R_\ep^{\widetilde{\kappa}} = \lr{\ep^{\frac{3}{4 (\gamma-1)\widetilde{\kappa}}} R_\ep}^{\widetilde{\kappa}}\leq c\lr{\ep  |\log \ep |}^{\frac{3}{4 (\gamma-1)}} \rightarrow 0 \quad \mbox{as} \quad \e \to 0
\end{equation}
and 
\begin{equation}\label{conv_Re_1}
    \ep^{2\beta} \left(\ep \exp{(-R_\ep^2)}\right)^{\frac{2(\alpha-\gamma)}{\gmhalf}}\leq \ep^2 \exp{(|\log \ep|^{2\theta })} \leq \ep,
\end{equation}
where $\beta$ is appeared in \eqref{app_eq2}. Also, we have \eqref{ep-R-1}, and more precisely, $\ep^\beta R_\ep^{1-\alpha}\to 0$ as $\ep \rightarrow 0$, which is an important hypothesis in Lemma \ref{lem:free-energy-lsc}.

%
%
%
%
%
%

\subsubsection*{Approximation of the initial density}
Here we provide several convergences related to the initial density $\rho^\e_0$ in the lemma below.
\begin{lemma}\label{id-l1} 
Let $R_\e$ be given as in \eqref{choiceR_ep}. 
Then, we have the following estimates and convergences: 
   \begin{gather}
        \inf_{x \in \Omega_\e} \vr_0^\ep(x) \geq c \left(\ep \exp{(-R_\ep^2)}\right)^{\frac{1}{\gmhalf}}  >0  \label{lb}\\
		\vr_0^{\ep} \rightarrow \rho_0 \text{ strongly in }  L^1\cap L^\gamma(\R),\label{gam-1}\\
		|x|^{\kappa+2}\vr_0^{\ep}  \rightarrow |x|^{\kappa+2}\rho_0 \text{ strongly in } L^1(\R). \label{moment}
   \end{gather}
\end{lemma}	
\begin{proof}
By definition of $\vr_0^\ep$, \eqref{lb} is clear. 

We observe that $\rho_0 \in L^p(\R) \subset L^1 \cap L^\infty(\R)$ for all $p \in [1,\infty]$, which together with \eqref{ini_app_app}, gives $ \vr_0^\gmhalf  \in H^1(\R)$. Thus, from the definition of $\tvr_0^{\ep}$, we deduce
\begin{align}\nonumber
\begin{aligned}
& (\tvr_0^{\ep})^{\gamma-\frac{1}{2}} \rightarrow \vr_0^{\gamma-\frac{1}{2}} \text{ strongly in } L^p(\R) \text{ for all } p\geq 1 \text{ and } \cr
&\partial_x ((\tvr_0^{\ep})^{\gamma-\frac{1}{2}}) \rightarrow \partial_x(\vr_0^{\gamma-\frac{1}{2}}) \text{ strongly in } L^2(\R). 
\end{aligned}
\end{align}

Now we consider two cases:  ($\gmhalf\leq 1$) and  ($\gmhalf > 1$).

\medskip

\paragraph{\bf (Case I: $\gmhalf\leq 1$)} Note that 
	\begin{align}\label{ipd-i1}
		|a^p-b^p| \leq 2^p |a-b| |a^{p-1} +b^{p-1}| \text{ for }a,b>0  \text{ and }p\geq 1.
	\end{align}
This together with noticing $\frac{1}{\gamma - \frac12} \geq 1$ gives
\begin{align*}
&\vert \widetilde{\vr}_0^{\ep} -  \vr_0\vert \\
&\quad = \levert \lr{\rho_0^{\gmhalf} \ast \eta_{\ep^{\gmhalf}} +\ep e^{-x^2} }^{\frac{1}{\gmhalf}} - \vr_0 \rivert\\
&\quad \leq 2^{\frac{1}{\gamma - \frac12} } \levert \lr{\rho_0^{\gmhalf} \ast \eta_{\ep^{\gmhalf}} +\ep e^{-x^2} } - \vr_0^{\gmhalf} \rivert  \levert \lr{\rho_0^{\gmhalf} \ast \eta_{\ep^{\gmhalf}} +\ep e^{-x^2}}^{\frac{1}{\gmhalf}-1} + \vr_0^{\frac{1}{\gmhalf}-1} \rivert.
\end{align*}
Thus, we get for any $p\in[1,\infty]$,
\[
\Vert \widetilde{\vr}_0^{\ep} -  \vr_0\Vert_{L^p} \leq c\|\rho_0^{\gmhalf} \ast \eta_{\ep^{\gmhalf}} -  \vr_0^{\gmhalf}\|_{L^p} + c\e \to 0\qquad\text{as $\e \to 0$}.
\]
In particular, we obtain
\[
			\tvr_0^{\ep} \rightarrow \rho_0 \text{ in } L^1\cap L^\gamma(\R),
\]
and subsequently
 \begin{equation}\label{conv_z0}
Z_\e \to 1 \quad \mbox{as} \quad \e \to 0.
 \end{equation}
Then, by using the above convergences, we conclude that
\[
			\vr_0^{\ep} \rightarrow \rho_0 \text{ in } L^1\cap L^\gamma(\R) \quad \mbox{as} \quad \e \to 0.
\]

Next, we show the convergence \eqref{moment}. We first observe that 
 \begin{equation}\label{conv_app0}
 \ds \lr{|x|^{{(\kappa+2)}{(\gmhalf)}}\vr_0^{\gmhalf} } \ast \eta_{\ep^{\gmhalf}} \rightarrow |x|^{{(\kappa+2)}{(\gmhalf)}}\vr_0^{\gmhalf} \text{ strongly in } L^{\frac{1}{\gmhalf}}(\R)  \quad \mbox{as} \quad \e \to 0
 \end{equation}
 due to $\ds |x|^{\kappa+2}\vr_0 \in L^1(\R)$. We then claim that 
\begin{equation}\label{claim_app}
\lim_{\e \to 0}\ds \levertl \lr{|x|^{{(\kappa+2)}{(\gmhalf)}}\vr_0^{\gmhalf} } \ast \eta_{\ep^{\gmhalf}} -|x|^{{(\kappa+2)}{(\gmhalf)}} \lr{\vr_0^{\gmhalf}  \ast \eta_{\ep^{\gmhalf}} } \rivertl_{L^{\frac{1}{\gmhalf}}} = 0.
\end{equation}
Note that if \eqref{claim_app} holds, it follows from \eqref{conv_app0} that 
 \[
 \ds |x|^{{(\kappa+2)}{(\gmhalf)}} \lr{\vr_0^{\gmhalf}  \ast \eta_{\ep^{\gmhalf}} } \rightarrow |x|^{{(\kappa+2)}{(\gmhalf)}}\vr_0^{\gmhalf}  \mbox{ strongly in }   L^{\frac{1}{\gmhalf}}(\R) \quad \mbox{as} \quad \e \to 0.
 \] 
Furthermore, by using \eqref{ipd-i1} and the fact that 
\[
\ds \levertl |x|^{\kappa+2} \lr{\ep e^{-x^2} }^{\frac{1}{\gmhalf}} \rivertl_{L^1} \leq c\e^{\frac1{\gamma - \frac12}},
\]
we find
 \begin{align*}
|x|^{\kappa+2}\tvr_0^{\ep} \rightarrow |x|^{\kappa+2}\rho_0 \text{ in } L^1(\R).
\end{align*}
Then again we use \eqref{conv_z0} and monotone convergence theorem to have  \eqref{moment}.

\begin{proof}[Proof of Claim \eqref{claim_app}] Note that
\[
| |x-y|^p -|x|^p | \leq c \lr{|y|^p + |y| |x-y|^{p-1}} \quad \mbox{for $p > 1$}.
\]
Since $(\kappa + 2)(\gamma - \frac12) > 1$ and $\|\eta_{\e^{\gamma - \frac12}}\|_{L^1} = 1$, we estimate
\begin{align*}
&	\intO{\levert \lr{|x|^{{(\kappa+2)}{(\gmhalf)}}\vr_0^{\gmhalf} } \ast \eta_{\ep^{\gmhalf}} -|x|^{{(\kappa+2)}{(\gmhalf)}} \lr{\vr_0^{\gmhalf}  \ast \eta_{\ep^{\gmhalf}} } \rivert^{\frac{1}{\gmhalf}}}\\
&\quad \leq \intO{ \lr{\intOy{\levert \lr{|x-y|^{{(\kappa+2)}{(\gmhalf)}}-|x|^{{(\kappa+2)}{(\gmhalf)}}} \vr_0^{\gmhalf}(x-y)  \eta_{\ep^{\gmhalf}}(y)  \; \rivert}}^{\frac{1}{\gmhalf}}}\\
&\quad \leq c \int_\R \left( \int_\R \left( |y|^{(\kappa + 2)(\gamma - \frac12)} + |y| |x-y|^{(\kappa + 2)(\gamma - \frac12) - 1} \right) \rho^{\gamma - \frac12}_0(x-y) \eta_{\e^{\gamma - \frac12}}(y)\d y\right)^{\frac1{\gamma - \frac12}} \d x\cr
&\quad \leq c \iint_{\R  \times \R} \left( |y|^{(\kappa + 2) } + |y|^{\frac1{\gamma - \frac12}} |x-y|^{(\kappa + 2) - \frac1{\gamma - \frac12}} \right) \rho_0(x-y) \eta_{\e^{\gamma - \frac12}}(y)\d y  \d x\cr
&\quad =: I + II,
\end{align*}
due to Jensen's inequality. Here, 
\begin{align*}
|I| \leq c \e^{(\kappa + 2)(\gamma - \frac12)}\iint_{\R \times \R}   \rho_0(x-y) \eta_{\e^{\gamma - \frac12}}(y)\d y \d x = c \e^{(\kappa + 2)(\gamma - \frac12)}.
\end{align*}
For $II$, we split it into two terms:
\begin{align*}
|II| &\leq c \e \iint_{\R  \times \R}  |x-y|^{(\kappa + 2) - \frac1{\gamma - \frac12}} ({\bf 1}_{|x-y| \leq 1} + {\bf 1}_{|x-y| \geq 1})\rho_0(x-y) \eta_{\e^{\gamma - \frac12}}(y)\d y  \d x  =: II_1 + II_2.
\end{align*}
Note that $(\kappa + 2)(\gamma - \frac12) - 1 > 0$ and thus
\[
II_1 \leq c\e \iint_{\R \times \R}   \rho_0(x-y) \eta_{\e^{\gamma - \frac12}}(y)\d y \d x = c\e
\]
and
\[
II_2 \leq c\e \iint_{\R \times \R} |x-y|^{\kappa + 1}\rho_0(x-y)  \eta_{\e^{\gamma - \frac12}}(y) \d y  \d x = c\e \|\rho_0\|_{L^1_{\kappa + 1}}.
\]
Hence, by combining all of the above estimates, we conclude \eqref{claim_app}.
\end{proof}

\medskip

\paragraph{\bf (Case II: $\gmhalf > 1$)} Since 
\begin{align}\label{ipd-i2}
|a^{\frac{1}{p}}-b^{{\frac{1}{p}}}| \leq |a-b|^{\frac{1}{p}} \text{ for }a,b>0 \text{ and } p\geq 1,
\end{align} 
we obtain
\begin{equation}\label{idp-c2-1}
\vert \widetilde{\vr}_0^{\ep} -  \vr_0\vert = \levert \lr{\rho_0^{\gmhalf} \ast \eta_{\ep^{\gmhalf}} +\ep e^{-x^2} }^{\frac{1}{\gmhalf}} - \vr_0 \rivert \leq \levert \lr{\rho_0^{\gmhalf} \ast \eta_{\ep^{\gmhalf}} +\ep e^{-x^2} } - \vr_0^{\gmhalf} \rivert^{{\frac{1}{\gmhalf}} }. 
\end{equation}
This together with $\widetilde\rho^\e_0,\, \rho_0 \in L^\infty(\R)$ implies that for any $p \geq \gamma - \frac12$
\[
\Vert \widetilde{\vr}_0^{\ep} -  \vr_0\Vert_{L^p}  \rightarrow 0 \quad \text{as} \quad \ep \rightarrow 0.
\]
In particular, we get
\[
\tvr_0^{\ep} \rightarrow \rho_0 \text{ in } L^\gamma(\R).
\]

Next, we claim that
   \begin{align*}
      \lim_{\e \to 0}\Vert (\tvr_0^{\ep} -  \rho_0)\mathbf{1}_{{\overline{\Omega}_\ep}}\Vert_{L^1} =0.
   \end{align*}
For this, we use \eqref{idp-c2-1} to estimate
    			\begin{align*}
				\intM{|\widetilde{\vr}_0^{\ep} -  \vr_0| } & \leq  \intM{\levert \lr{\rho_0^{\gmhalf} \ast \eta_{\ep^{\gmhalf}} +\ep e^{-x^2} } - \vr_0^{\gmhalf} \rivert^{{\frac{1}{\gmhalf}} }}  \\
				&\leq \delta^{\frac{\gmhalf}{\gamma-1}} \intM{1 } +  \frac{c}{\delta^{2(\gmhalf)}} \intM{  \levert \lr{\rho_0^{\gmhalf} \ast \eta_{\ep^{\gmhalf}} +\ep e^{-x^2} } - \vr_0^{\gmhalf} \rivert^2 } \\
				& \leq 2\delta^{\frac{\gmhalf}{\gamma-1}} R_\ep + \frac{\ep^{2}}{\delta^{2(\gmhalf)} } + \frac{c}{\delta^{2(\gmhalf)}} \Vert \rho_0^{\gmhalf} \ast \eta_{\ep^{\gmhalf}}  - \vr_0^{\gmhalf} 
				\Vert_{L^2}^2,
			\end{align*}
where $\delta >0$ will be determined later. On the other hand, we observe
			\begin{align*}
				 \Vert \rho_0^{\gmhalf} \ast \eta_{\ep^{\gmhalf}}  - \vr_0^{\gmhalf} 
				\Vert_{L^2} \leq C \ep^{\gmhalf}  \Vert  \vr_0^{\gmhalf} 
				\Vert_{H^1},
			\end{align*}
and thus
			\begin{align*}
				\intM{\vert \widetilde{\vr}_0^{\ep} -  \vr_0\vert } 
				& \leq 2\delta^{\frac{\gmhalf}{\gamma-1}} R_\ep + \frac{\ep^2}{\delta^{2(\gmhalf)} } + \frac{c}{\delta^{2(\gmhalf)}} \ep^{2(\gmhalf)} \leq 2\delta^{\frac{\gmhalf}{\gamma-1}} R_\ep + \frac{c\ep^2}{\delta^{2(\gmhalf)} }
			\end{align*}
			due to $\e < 1$.
Then, we choose $\delta = \ep^{\frac{3}{4(\gmhalf)}} > 0$ to get from \eqref{conv_Re_0} that
			\begin{align*}
				\intM{\vert \widetilde{\vr}_0^{\ep} -  \vr_0\vert } 
				& \leq 2\ep^{\frac{3}{4 (\gamma-1)}} R_\ep + c \ep^{1/2} \leq 2\ep^{\frac{3}{4 (\gamma-1)}} R_\ep^{\tilde\kappa} + c \ep^{1/2} \to 0 \quad \mbox{as} \quad \e \to 0.
			\end{align*}
			On the other hand, due to 	the monotone convergence theorem, we find	
			\begin{align*}
				\rho_0\mathbf{1}_{\overline{\Omega}_\ep}  \rightarrow \rho_0 \text{ in } L^1(\R). 
			\end{align*}
Hence, the triangle inequality gives $ \tvr_0^{\ep}\mathbf{1}_{\overline{\Omega}_\ep} \rightarrow \rho_0 $ in $L^1(\R)$. In particular, this shows $Z_\e \to 1$ and subsequently, we have $\rho^\e_0 \to \rho_0$ in $L^1 \cap L^\gamma(\R)$ as $\e \to 0$.
			
For the convergence \eqref{moment}, similarly as before, by monotone convergence theorem, we first observe	
			\begin{align*}
				|x|^{2+\kappa} \rho_0\mathbf{1}_{{\overline{\Omega}_\ep}}  \rightarrow |x|^{2+\kappa} \rho_0 \text{ in } L^1(\R)  \quad \mbox{as} \quad \e \to 0.
			\end{align*}
Then, for any $\delta > 0$, we again use \eqref{idp-c2-1} to estimate			
			\begin{align*}
				&\intM{|x|^{2+\kappa} \vert \widetilde{\vr}_0^{\ep} -  \vr_0\vert } \\
				&\quad  \leq \intM{   |x|^{2+\kappa} \levert  \lr{\rho_0^{\gmhalf} \ast \eta_{\ep^{\gmhalf}} +\ep e^{-x^2} } - \vr_0^{\gmhalf} \rivert^{{\frac{1}{\gmhalf}} }   } \\
				&\quad \leq \delta^{\frac{\gmhalf}{\gamma-1}} \intM{|x|^{(2+\kappa) \frac{\gmhalf}{\gamma-1} }} + \frac{c}{\delta^{2(\gmhalf)}} \intM{  \levert \lr{\rho_0^{\gmhalf} \ast \eta_{\ep^{\gmhalf}} +\ep e^{-x^2} } - \vr_0^{\gmhalf} \rivert^2 } \\
				&\quad  \leq c\delta^{\frac{\gmhalf}{\gamma-1}} R_\ep^{(2+\kappa) \frac{\gmhalf}{\gamma-1} +1} + \frac{\ep^{2}}{\delta^{2(\gmhalf)} } + \frac{c}{\delta^{2(\gmhalf)}} \Vert \rho_0^{\gmhalf} \ast \eta_{\ep^{\gmhalf}}  - \vr_0^{\gmhalf} 
				\Vert_{L^2}^2\cr
				&\quad  \leq c\delta^{\frac{\gmhalf}{\gamma-1}} R_\ep^{(2+\kappa) \frac{\gmhalf}{\gamma-1} +1} + \frac{c\ep^{2}}{\delta^{2(\gmhalf)} }.
			\end{align*}
By choosing $\delta = \ep^{\frac{3}{4(\gmhalf)}}$, we have		
			\begin{align*}
				\intM{ |x|^{2+\kappa}\vert \widetilde{\vr}_0^{\ep} -  \vr_0\vert } 
				& \leq 2\ep^{\frac{3}{4 (\gamma-1)}} R_\ep^{\widetilde{\kappa}} + c \ep^{1/2} \to 0 \quad \mbox{as} \quad \e \to 0,
			\end{align*}
			where $\widetilde{\kappa} =(2+\kappa) \frac{\gmhalf}{\gamma-1} +1 $ due to \eqref{conv_Re_0}. Then again by using $Z_\e \to 1$ as $\e \to 0$, we conclude the convergence \eqref{moment}.
 \end{proof}

%
%
%
%
%

\subsubsection*{Approximation of initial velocity} 
In the following lemma, we show the convergence of kinetic energy and uniform-in-$\e$ bound estimate on the velocity moments.
\begin{lemma}\label{id-l2}
The following convergence and bound estimate hold:
\[
\lim_{\e \to 0}\int_\R \rho^\e_0 |u^\e_0|^2 \d x = \int_\R \rho_0 |u_0|^2 \d x \quad \mbox{and} \quad \limsup_{\e \to 0}\int_\R \rho^\e_0 |u^\e_0|^{2+\kappa} \d x \leq \int_\R \rho_0 |u_0|^{2+\kappa} \d x.
\]
\end{lemma}
\begin{proof} 
By using the triangle inequality, H\"older's inequality, and \eqref{ipd-i2}, we obtain
    \begin{align*}
        \Vert w_0^\ep - w_0 \Vert_{L^2} &\leq  \Vert {(\rho_0^\ep)}^{\frac{1}{2}-\frac{1}{2+\kappa}} \lr{\tilde{w}_0^\ep -\tilde{w}_0} \Vert_{L^2} +  \|({(\rho_0^\ep)}^{\frac{1}{2}-\frac{1}{2+\kappa}}-{(\rho_0)}^{\frac{1}{2}-\frac{1}{2+\kappa}})\tilde{w}_0 \Vert_{L^2} \\
        &\leq \Vert {(\rho_0^\ep)}^{\frac1p} \Vert_{L^p}\Vert \tilde{w}_0^\ep -\tilde{w}_0 \Vert_{L^{2+k}}+ \Vert {(\rho_0^\ep)}^{\frac1p}-{(\rho_0)}^{\frac1p} \Vert_{L^p}\Vert \tilde{w}_0  \Vert_{L^{2+\kappa}}\cr
        &\leq \|\rho^\e_0\|_{L^1}^\frac1p \Vert \tilde{w}_0^\ep -\tilde{w}_0 \Vert_{L^{2+k}} + \|\rho^\e_0 - \rho_0\|_{L^1} \Vert \tilde{w}_0  \Vert_{L^{2+\kappa}}\
    \end{align*} 
    where $ \frac{1}{p}= \frac{1}{2}-\frac{1}{2+\kappa}$. This together with Lemma \ref{id-l1} and \eqref{app_tw0} yields the convergence of initial kinetic energy.
    
For the uniform bound estimate on the velocity moments, by definition of $u^\e_0$, we easily find
\[
\int_\R \rho^\e_0 |u^\e_0|^{2+\kappa} \d x \leq \int_\R |\tilde w^\e_0|^{2+\kappa} \d x.
\]
Thus, we get from \eqref{app_tw0} that
\[
\limsup_{\e \to 0}\int_\R \rho^\e_0 |u^\e_0|^{2+\kappa} \d x \leq \int_\R \rho_0 |u_0|^{2+\kappa} \d x.
\]
This completes the proof.
\end{proof}

%
%
%
%
%

We are now in a position to give the details of proof for Proposition \ref{prop_convJ}.

\begin{proof}[Proof of Proposition \ref{prop_convJ}] 
As a direct consequence of Lemma \ref{id-l2}, we get
\[
\limsup_{\e \to 0}\int_\R \rho^\e_0 |u^\e_0|^{2+\kappa} \d x\leq \int_\R \rho_0 |u_0|^{2+\kappa} \d x.
\]
For the convergence of  $\calJ_{\tau,\lambda,\e}$, by definition of that, it suffices to show that 
\begin{equation}\label{conv_kin_free}
\int_\R \rho^\e_0 |u^\e_0|^2 \d x \to \int_\R \rho_0 |u_0|^2 \d x, \quad  \calF_\e(\rho^\e_0) \to \calF(\rho_0),  
\end{equation}
and
\begin{equation}\label{conv_kin_free2}
\intM{  \rho_0^\ep |u^\ep_0 + \pa_x \delta \calF_\ep(\rho_0^\ep)|^2} \to \intM{  \rho_0 |u_0 + \pa_x \delta \calF(\rho_0)|^2}
\end{equation}
as $\e \to 0$. Note that the convergences \eqref{conv_kin_free} can be directly obtained from the results of Lemmas \ref{id-l1} and \ref{id-l2}. For \eqref{conv_kin_free2}, we notice that 
\begin{align*}
    &\frac12\intM{  \rho_0^\ep |u^\ep_0 + \pa_x \delta \calF_\ep(\rho_0^\ep)|^2} \\
    &\quad = \frac{1}{2} \intM{ \rho_0^\ep |u^\ep_0 |^2 }+ \intM{ \rho_0^\ep u^\ep_0 \lr{\pa_x \delta \calF_\ep(\rho_0^\ep)} } + \frac{1}{2} \intM{ \rho_0^\ep |\pa_x \delta \calF_\ep(\rho_0^\ep) |^2},
\end{align*}
where
\begin{align*}
    \frac{1}{2}\intM{ \rho_0^\ep |\pa_x \delta \calF_\ep(\rho_0^\ep) |^2}&=  \frac{1}{2}\intM{ \rho_0^\ep |\pa_x \varphi(\rho_0^\ep ) |^2} + \frac{\ep^\beta \alpha}{\alpha-1} \intM{ \rho_0^\ep (\pa_x \varphi(\rho_0^\ep )) (\pa_x (\vr_0^\ep)^{\alpha-1} ) }\\ 
    & \quad + \ep^{2\beta} \frac12 \lr{\frac{\alpha}{\alpha-1}}^2  \intM{ \rho_0^\ep |\pa_x (\vr_0^\ep)^{\alpha-1}  |^2 }.
\end{align*}
Here we use \eqref{lb} to estimate the third term on the right-hand side of the above equality as 
\begin{align*}
    \ep^{2\beta} \intM{ \vr_0^\ep |\pa_x (\vr_0^\ep)^{\alpha-1}  |^2 }& = c_\alpha Z_\ep^{1 - 2\gamma } \ep^{2\beta}  \intM{ (\vr_0^\ep)^{2(\alpha-\gamma)} |\pa_x (\tvr_0^\ep)^{\gamma-\frac12}  |^2 }\cr
    & \leq c_\alpha  \ep^{2\beta}\left(\ep \exp{(-R_\ep^2)}\right)^{\frac{2(\alpha-\gamma)}{\gmhalf}} \to 0 \quad \mbox{as} \quad  \e \to 0,
\end{align*}
due to \eqref{conv_Re_1} and $Z_\e \to 1$ as $\e \to 0$. This together with the convergence results of Lemmas \ref{id-l1} and \ref{id-l2} concludes the desired result.
\end{proof}

%
%
%
%
%

\section*{Acknowledgement}
The research of YPC was supported by NRF grant no.\ 2022R1A2C1002820. The research of NC and EZ  was supported by the EPSRC Early Career Fellowship  EP/V000586/1. Also, the work of NC was partly supported by the ``Excellence Initiative Research University (IDUB)" program at the University of Warsaw. OT acknowledges support from NWO grant OCENW.M.21.012.



\end{document}